\documentclass[12pt,a4paper,twoside,reqno]{amsart}
\allowdisplaybreaks

\usepackage[all,cmtip]{xy}
\usepackage[T1]{fontenc}
\usepackage{amsmath,amscd}
\usepackage{amssymb,amsfonts}
\usepackage[driver=pdftex,margin=3cm,heightrounded=true,centering]{geometry}
\usepackage{mathtools}
\usepackage{tensor}
\usepackage{url}
\usepackage[colorlinks=true,linkcolor=blue,citecolor=blue]{hyperref}
\usepackage{slashed}

\usepackage{graphicx}

\usepackage{amsthm}

\theoremstyle{plain}

\newtheorem{thm}{Theorem}[section]
\newtheorem*{thm*}{Theorem}
\newtheorem{prop}[thm]{Proposition}
\newtheorem{lemma}[thm]{Lemma}
\newtheorem{cor}[thm]{Corollary}

\theoremstyle{definition}
\newtheorem{dfn}[thm]{Definition}

\theoremstyle{remark} 

\numberwithin{equation}{section}

\newcommand{\alpheqn}[1][\relax]{
     \refstepcounter{equation}
     \if#1\relax \relax
       \else \label{#1}
     \fi  
     \setcounter{saveeqn}{\value{equation}}%
    \setcounter{equation}{0}%
    \renewcommand{\theequation}{\thealphequation}}
\newcommand{\reseteqn}{\setcounter{equation}{\value{saveeqn}}%
     \renewcommand{\theequation}{\thearabicequation}}

\providecommand{\mathscr}{\mathcal} 
\IfFileExists{mathrsfs.sty}{
\usepackage{mathrsfs}}         
   {
     \IfFileExists{eucal.sty}{
        \usepackage[mathscr]{eucal}} 
   {
   }
}

\newcommand{\vertiii}[1]{{\left\vert\kern-0.25ex\left\vert\kern-0.25ex\left\vert #1 
    \right\vert\kern-0.25ex\right\vert\kern-0.25ex\right\vert}}
\newcommand{\Bvert}[1]{{\Big\vert\kern-0.25ex\Big\vert\kern-0.25ex\Big\vert #1 
    \Big\vert\kern-0.25ex\Big\vert\kern-0.25ex\Big\vert}}
\newcommand{\bvert}[1]{{\big\vert\kern-0.25ex\big\vert\kern-0.25ex\big\vert #1 
    \big\vert\kern-0.25ex\big\vert\kern-0.25ex\big\vert}}
\newcommand{\nvert}[1]{{\vert\kern-0.25ex\vert\kern-0.25ex\vert #1 
    \vert\kern-0.25ex\vert\kern-0.25ex\vert}}

\renewcommand{\leq}{\leqslant}
\renewcommand{\geq}{\geqslant}

\newcommand{\cd}{\cdot}

\newcommand{\ot}{\otimes}
\newcommand{\hot}{\widehat \otimes}

\newcommand{\op}{\oplus}

\newcommand{\ci}{\circ}

\newcommand{\ti}{\times}

\newcommand{\zz}{\mathbb{Z}}

\newcommand{\al}{\alpha}

\newcommand{\ga}{\gamma}

\newcommand{\de}{\delta}
\newcommand{\De}{\Delta}
\newcommand{\ep}{\varepsilon}

\newcommand{\io}{\iota}
\newcommand{\ka}{\kappa}
\newcommand{\la}{\lambda}

\newcommand{\si}{\sigma}

\newcommand{\ov}{\overline}
\newcommand{\C}[1]{\mathcal{#1}}

\newcommand{\T}[1]{\textup{#1}}

\newcommand{\B}[1]{\mathbb{#1}}

\newcommand{\fork}[2]{\left\{ \begin{array}{#1} #2 \end{array} \right.}

\newcommand{\ma}[2]{\left(\begin{array}{#1} #2 \end{array} \right)}

\newcommand{\su}{\subseteq}

\newcommand{\q}{\qquad}
\newcommand{\qq}{\qquad \qquad}

\newcommand{\binn}[1]{\big\langle #1 \big\rangle}

\newcommand{\sem}{\setminus}

\makeatletter
\@namedef{subjclassname@2020}{%
  \textup{2020} Mathematics Subject Classification}
\makeatother

\begin{document}
\title[Spectral localizers in $KK$-theory]{Spectral localizers in $KK$-theory}


\author{Jens Kaad}

\address{Department of Mathematics and Computer Science,
The University of Southern Denmark,
Campusvej 55, DK-5230 Odense M,
Denmark}

\email{kaad@imada.sdu.dk}


%
%
%

\begin{abstract}
We study the index homomorphism of even $K$-groups arising from a class in even $KK$-theory via the Kasparov product. Due to the seminal work of Baaj and Julg, under mild conditions on the $C^*$-algebras in question such a class in $KK$-theory can always be represented by an unbounded Kasparov module. We then describe the corresponding index homomorphism of even $K$-groups in terms of spectral localizers. This means that our explicit formula for the index homomorphism does not depend on the full spectrum of the abstract Dirac operator $D$, but rather on the intersection between this spectrum and a compact interval. The size of this compact interval does however reflect the interplay between the $K$-theoretic input and the abstract Dirac operator. Since the spectral projections for $D$ are not available in the general context of Hilbert $C^*$-modules we instead rely on certain continuous compactly supported functions applied to $D$ to construct the spectral localizer. In the special case where even $KK$-theory coincides with even $K$-homology, our work recovers the pioneering work of Loring and Schulz-Baldes on the index pairing. 
\end{abstract}

\maketitle
\tableofcontents

\section{Introduction}
The spectral localizer approach to the index pairing was pioneered by Terry Loring and Hermann Schulz-Baldes in the papers \cite{LoSc:FVC,LoSc:SLE} and their work was soon followed by a range of applications to physics, see the non-exhaustive list \cite{FPL:AWT,CeLo:LIT,CeLo:CPT,FrGr:TZM}. One of the key points is that the spectral localizer can be identified with an invertible hermitian matrix over $\B C$ and half the signature of this matrix agrees with the output of the index pairing. In practice, this allows for an efficient numerical implementation of the index problem. At a more theoretical level spectral localizers also feature as motivation for the recent developments on $K$-theory for operator systems \cite{Sui:GKO,Sui:HKO} led by Walter van Suijlekom.

The index pairing comes in two different versions referred to as even and odd index pairings. The spectral localizer approach to the index pairing is carefully described in both cases in the recent book \cite{DSW:SF} and extensions to the semifinite setting is given in \cite{ScSt:SLS}. Both of these sources apply spectral flow as an integral part of their argumentation. In this paper we focus on the even case and we therefore proceed by explaining the spectral localizer approach in this setting as it appears in \cite{LoSc:SLE} and \cite[Chapter 10]{DSW:SF}. Notice that the conditions stated here below have been further weakened in the recent paper \cite{CeSc:DLT}. 

Consider an invertible selfadjoint bounded operator $H$ acting on a separable Hilbert space $G$ together with an invertible closed unbounded operator $D_0 : \T{Dom}(D_0) \to G$ satisfying the following conditions:
\begin{enumerate}
\item The selfadjoint unbounded operator $D := \ma{cc}{0 & D_0^* \\ D_0 & 0}$ has compact resolvent;
\item $H \op H$ preserves the domain of $D$ and the commutator $[D, H \op H] : \T{Dom}(D) \to G \op G$ extends to a bounded operator on $G \op G$.
\end{enumerate}
Define the projection $Q := \frac{1 + H |H|^{-1}}{2}$ and record that $Q D_0 Q : Q \T{Dom}(D_0) \to Q G$ is an unbounded Fredholm operator and hence gives rise to the integer valued index
\[
\T{Index}(Q D_0 Q) := \T{Dim}_{\B C}\big( \T{Ker}( Q D_0 Q) \big) - \T{Dim}_{\B C}\big( \T{Coker}( Q D_0 Q) \big) .
\]
One of the main results in \cite[Chapter 10]{DSW:SF}, see also \cite{LoSc:SLE}, then states that the above index can be computed by means of the spectral localizers which we now review. For every $\rho > 0$, let $P_\rho := 1_{[-\rho,\rho]}(D)$ denote the spectral projection associated with $D$ and the closed interval $[-\rho,\rho]$. For every $\ka > 0$, define the spectral localizer
\[
L_{\ka,\rho} := P_\rho \ma{cc}{H & \ka D_0^* \\ \ka D_0 & -H} P_\rho : P_\rho (G \op G) \to P_\rho(G \op G) 
\]
and notice that the Hilbert space $P_\rho (G \op G)$ is in fact finite dimensional (since $D$ has compact resolvent). By \cite[Theorem 10.3.1]{DSW:SF}, for $\ka$ small enough and $\rho$ large enough, it holds that $L_{\ka,\rho}$ is invertible (and selfadjoint) and we have the identity
\begin{equation}\label{eq:signintro}
\frac{1}{2} \T{sign}(L_{\ka,\rho}) = \T{Index}(Q D_0 Q) 
\end{equation}
between half the signature of the spectral localizer and the index. Remark that \cite[Theorem 10.3.1]{DSW:SF} also quantifies what it means for $\ka$ to be ``small enough'' and $\rho$ to be ``large enough'' in terms of the spectral gap $\| H^{-1} \|^{-1}_\infty$ and the operator norm of the closure of the commutator $[D,H \op H]$. 

Let us continue by explaining what the identity in \eqref{eq:signintro} has to do with the index pairing between $K$-theory and $K$-homology, see \cite{HiRo:AKH}. To this end, let $C^*(1,H)$ denote the unital $C^*$-algebra generated by $H$ and record that the projection $Q$ determines a class in the $K$-theory group $K_0\big( C^*(1,H) \big)$. Moreover, it holds that the phase $D|D|^{-1}$ determines a class in the $K$-homology group $K^0\big( C^*(1,H) \big)$. Applying the index pairing to these classes we arrive at the right hand side of \eqref{eq:signintro}, meaning that 
\[
\T{Index}(Q D_0 Q) = \binn{ [Q], [G \op G,\pi, D|D|^{-1}]},
\]
where $\pi$ is the diagonal representation of $C^*(1,H)$ on the $\zz/2\zz$-graded Hilbert space $G \op G$.

It turns out that the index pairing between $K$-theory and $K$-homology has a far reaching generalization which is implemented by the Kasparov product \cite{Kas:OFE,CoSk:LIF}. Instead of a class in $K$-homology one may consider a class in $KK$-theory and the Kasparov product induces index homomorphisms between $K$-groups. In this paper, we begin with a class in even $KK$-theory and describe the corresponding index homomorphism of even $K$-groups in terms of a spectral localizer construction. Our results are therefore related to, but different from, the results obtained in the recent paper \cite{LiMe:OSL}. An important difference is that Li and Mesland start out with a class in odd $KK$-theory and describe the index homomorphism from odd $K$-theory to even $K$-theory. Moreover, in the present text, we do not impose any conditions regarding spectral decompositions of our Hilbert $C^*$-modules, see \cite[Definition 4.1 and Remark 4.4]{LiMe:OSL}. Notice however that the main result of Li and Mesland regarding the $E$-theoretic index homomorphism works without assuming the existence of spectral decompositions, see \cite[Theorem 3.6]{LiMe:OSL}, 

As mentioned above, our starting point is a class $[X,\pi,F]$ which belongs to the even $KK$-group $KK_0(A,B)$ associated to a pair of $C^*$-algebras $(A,B)$ and our aim is to compute the associated index homomorphism
\[
\binn{\bullet, [X,\pi,F]} : K_0(A) \to K_0(B)
\]
by means of spectral localizers. In particular, our results should recover the identity in \eqref{eq:signintro} in the special case where $B = \B C$ (using the isomorphism $K_0(\B C) \cong \B Z$ which is implemented by the operator trace). Assuming that $A$ is separable and that $B$ has a countable approximate identity, a fundamental result due to Baaj and Julg says that the class $[X,\pi,F]$ can be represented by an even unbounded Kasparov module $(X,\pi,D)$ where the abstract Dirac operator $D$ has compact resolvent, see \cite[Proposition 2.3]{BaJu:TBK}. At the $K$-theoretic level we are interested in an even selfadjoint invertible bounded operator $H : X \to X$ which is assumed to be a Lipschitz operator with respect to $D$ in the following sense:
\begin{itemize}
\item $H$ preserves the domain of $D$ and the commutator $[D,H] : \T{Dom}(D) \to X$ extends to a bounded adjointable operator. 
\end{itemize}
We remark in passing that, given a class in the $K$-theory of $A$, a representative satisfying this kind of Lipschitz condition can always be found, see \cite[Th\'eor\`eme A.2.1]{Bos:OKA}.

Let us proceed by constructing the relevant spectral localizer. To ease the exposition somewhat, assume that the Hilbert $C^*$-module $X$ agrees with the standard module $\ell^2(\B N,B)$ and recall that the compact operators on $\ell^2(\B N,B)$ can be identified with the (minimal) tensor product $B \ot \B K$ between $B$ and the compact operators on the Hilbert space $\ell^2(\B N)$. In particular, we may identify the two $K$-groups $K_0\big(\B K(X)\big)$ and $K_0(B)$. Choose an even smooth function $\phi : \B R \to [0,1]$ satisfying the following two conditions:
\begin{enumerate}
\item The support of $\phi$ is contained in $[-1,1]$ and $\phi(x) = 1$ for all $x \in [-1/2,1/2]$;
\item $\phi(x_0) \geq \phi(x_1)$ for all $0 \leq x_0 \leq x_1$. 
\end{enumerate}
Applying the continuous functional calculus (see e.g. \cite[Theorem 10.9]{Lan:HCM}) we obtain an even bounded adjointable operator $\Phi_\rho := \phi(D/\rho)$ for every $\rho > 0$. For every $\ka > 0$, define the spectral localizer
\[
L_{\ka,\rho} := \Phi_\rho \ga H \Phi_\rho + \ka \cd \Phi_{2\rho} D \Phi_{2\rho} - (1 - \Phi_{2\rho}^4)^{1/2} \ga ,
\]
where $\ga : X \to X$ denotes the $\zz/2\zz$-grading operator on $X$. For $\ka$ small enough and $\rho$ large enough we show in this text that $L_{\ka,\rho}$ is invertible and hence, after applying excision in $K$-theory \cite[Theorem 1.5.9]{Ros:AKA}, we get that the pair $(-\ga, L_{\ka,\rho})$ defines a class in the even $K$-theory of the compact operators on $X$, $K_0\big(\B K(X)\big)$. The main result of this paper provides a spectral localizer description of the index homomorphism:

\begin{thm}\label{t:localizerintro}
  For $\ka$ small enough and $\rho$ large enough, upon identifying $K_0\big(\B K(X)\big)$ with $K_0(B)$, we have the identity
  \[
  \big[(-\ga, L_{\ka,\rho})\big] = \binn{ [ (1 + H|H|^{-1})/2], [X,\pi,F]}
  \]
  inside the $K$-group $K_0(B)$.
\end{thm}

We are also able to quantify what it means for $\ka$ to be ``small enough'' and $\rho$ to be ``large enough''. This quantification involves the spectral gap $\| H^{-1} \|_\infty^{-1}$, the operator norm of the closure of the commutator $[D,H]$ as well as the $L^1$-norm of the Fourier transform of the derivative of $\phi$.

The strategy we follow to prove Theorem \ref{t:localizerintro} is novel (also in the case where the second argument $B$ is equal to $\B C$). Indeed, we start out by proving Theorem \ref{t:localizerintro} in the special case where $H^2$ equals the identity operator and the commutator $[D,H]$ is equal to zero. Afterwards, we easily reduce the proof to this special case by using the invariance of the $K$-theory class of the spectral localizer under selfadjoint bounded perturbations of $D$. To understand the advantage of this approach, let us briefly explain how it works in the Hilbert space setting (thus when $B = \B C$). Indeed, if $H = 2Q-1$ for a projection $Q = \ma{cc}{Q_+ & 0 \\ 0 & Q_-}$ and $[D,H]$ is zero, then the spectral localizer is invertible for all values of $\ka$ and $\rho$. Hence, choosing $\rho$ small enough we get that both $\Phi_\rho$ and $\Phi_{2\rho}$ agrees with the projection $P$ onto the kernel of $D = \ma{cc}{0 & D_0^* \\ D_0 & 0}$ and hence that
\[
L_{\ka,\rho} = P \ga H P - (1 - P) \ga .
\]
Letting $\T{TR}$ denote the operator trace it is then not difficult to verify the following identities
\[
\frac{1}{2} \T{sign}( P \ga H P) + \frac{1}{2}\T{sign}(P \ga P) = \T{TR}(\ga P Q P) = \T{Index}(Q_- D_0 Q_+)
\]
which establish the desired link between the spectral localizer and the index. Remark in this respect that the term $\frac{1}{2}\T{sign}(P \ga P)$ agrees with half the index of $D_0$. 

The structure of the present paper is as follows: In Section \ref{s:func} we revisit the continuous functional calculus for selfadjoint and regular unbounded operators on Hilbert $C^*$-modules focusing on the Cayley-transform. This section also contains commutator estimates which are crucial for the remaining text. In Section \ref{s:ktheory} we deepen the understanding of $K$-theory for selfadjoint invertible operators as it appears in \cite{GrSc:IPS,BoLo:KRU}. Our approach is based on the excision isomorphism between relative $K$-theory and $K$-theory of the ideal in question. In Section \ref{s:unbdd} we review some basic aspects of unbounded $KK$-theory together with a simple instance of the unbounded Kasparov product which is relevant for describing the index homomorphism. In Section \ref{s:index} we turn our attention to the well-known isomorphism between $KK_0(\B C,B)$ and $K_0(B)$ providing an explicit description of this isomorphism and its inverse. In Section \ref{s:spectral} and Section \ref{s:specprop} we introduce the spectral localizer in the context of Hilbert $C^*$-modules and study its invariance under various operations including homotopies and selfadjoint bounded perturbations. The core technical part of this paper is contained in Section \ref{s:compu} and Section \ref{s:specindex} where we relate the spectral localizer to the index homomorphism. In Section \ref{s:specproj} we study the special situation where the spectral projections for $D$ belong to the $C^*$-algebra of bounded adjointable operators on $X$ and we explain how our results recover the original results of Loring and Schulz-Baldes, see \cite{LoSc:SLE}, in the generality presented in \cite{DSW:SF}.

\subsection{Acknowledgements}
The author gratefully acknowledge the financial support
from the Independent Research Fund Denmark through grant no. 9040-00107B and
1026-00371B. This research is moreover part of the EU Staff Exchange project 101086394
"Operator Algebras That One Can See”.

The author would also like to thank Hermann Schulz-Baldes for our nice email conversation this summer and for pointing out a couple of important references on spectral localizers. As always, many thanks are due to Walter van Suijlekom and this time in particular for his mini-course during the 2024 INdAM conference ``Noncommutative Geometry and Applications'' in Cortona. This mini-course and subsequent discussions with Walter sparked the idea of studying spectral localizers in a $KK$-theoretic context. I would finally like to thank Koen van den Dungen and Bram Mesland for our conversation on spectral localizers during the 2025 NSeaG-workshop at the ICMS in Edinburgh. At this workshop I also benefitted a lot from Bram Mesland's talk on the $E$-theoretic approach to spectral localizers developed by Bram and his PhD-student Yuezhao Li.  

\subsection{Standing conventions}\label{ss:standing}
All Hilbert $C^*$-modules appearing in this text are right Hilbert $C^*$-modules. 

The $C^*$-algebra of bounded adjointable operators on a Hilbert $C^*$-module $X$ over a $C^*$-algebra $A$ is denoted by $\B L(X)$ and the notation $\| \cd \|_\infty$ refers to the operator norm on $\B L(X)$. The norm on $X$ is denoted by $\| \cd \|_X$. The compact operators on $X$, in the sense of Hilbert $C^*$-module theory, form a closed $*$-ideal $\B K(X)$ sitting inside $\B L(X)$.

We apply the notation $\ell^2(\B N,A)$ for the standard Hilbert $C^*$-module over the $C^*$-algebra $A$ so that elements in $\ell^2(\B N,A)$ are sequences $\{x_n\}_{n = 1}^\infty$ in $A$ satisfying that the series $\sum_{n = 1}^\infty x_n^* x_n$ converges in $C^*$-norm.

The unitalization of $A$ is denoted by $A^\sim$ and we view $A$ as a closed $*$-ideal inside $A^\sim$. Remark that if $\si : A \to B$ is a $*$-homomorphism with values in a unital $C^*$-algebra $B$, then $\si$ extends uniquely to a unital $*$-homomorphism $\si : A^\sim \to B$. Hence, applying $\si$ entry by entry we get a unital $*$-homomorphism $\si : M_n(A^\sim) \to M_n(B)$ for every $n \in \B N$.

%

For a non-empty locally compact Hausdorff space $M$ we write $C_b(M)$ for the unital $C^*$-algebra consisting of all bounded continuous complex valued functions on $M$. Inside $C_b(M)$ we have the $C^*$-subalgebra $C_0(M)$ of continuous complex valued functions on $M$ which vanish at infinity. The notation $C_c(M) \su C_0(M)$ refers to the norm-dense $*$-subalgebra of continuous functions with compact support. In the case where $M$ is compact we omit the subscripts and write $C(M)$ instead of $C_b(M) = C_0(M) = C_c(M)$. 

We define the Fourier transform of an $L^1$-function $f : \B R \to \B C$ by
\[
\widehat{f}(p) = \frac{1}{\sqrt{2\pi}} \int_{-\infty}^\infty f(x) e^{-ixp} \, dx \q p \in \B R .
\]
The $L^1$-norm of the $L^1$-function $f$ is denoted by $\| f \|_1 := \int_{-\infty}^\infty |f(x)| \, dx$.

\section{Continuous functional calculus}\label{s:func}
Throughout this section we consider a non-trivial Hilbert $C^*$-module $X$ over a $C^*$-algebra $A$. Moreover, our Hilbert $C^*$-module is equipped with a selfadjoint and regular unbounded operator $D : \T{Dom}(D) \to X$.

Let us clarify that $\T{Dom}(D) \su X$ is a norm-dense $A$-submodule and that selfadjointness and regularity amount to saying that $D$ is symmetric and satisfies that both $D + i$ and $D - i$ have image equal to $X$. We thus have a well-defined resolvent
$R_i := (D - i )^{-1}$ which is a bounded adjointable operator on $X$ with adjoint $R_i^* = (D + i)^{-1}$.

The continuous functional calculus associated to $D$ is described in \cite[Theorem 10.9]{Lan:HCM}, but see also \cite{Baa:MNB}, \cite[Theorem 1.5]{Wor:UAQ} and \cite{WoNa:OTF} for earlier accounts of this result. We now provide some further details on the main constructions. Our presentation differs from the one given in \cite{Lan:HCM} since we apply the Cayley transform instead of the bounded transform. Notice however that the continuous functional calculus presented in \cite{Wor:UAQ,WoNa:OTF} applies in the greater generality of normal and regular unbounded operators (and this requires the use of the bounded transform instead of the Cayley transform). Our exposition also relies on the unpublished work of Kustermann, see \cite{Kus:FCR}.

\begin{dfn}
  The \emph{spectrum} of the selfadjoint and regular unbounded operator $D$ consists of all elements $\la \in \B C$ where the $A$-linear map
\[
D - \la : \T{Dom}(D) \to X
\]
fails to be a bijection. The spectrum of $D$ is denoted by $\si(D) \su \B C$.
\end{dfn}

We remark that if $\la \notin \si(D)$, then the inverse of $D - \la : \T{Dom}(D) \to X$ followed by the inclusion $\T{Dom}(D) \su X$ is automatically a bounded operator $(D - \la)^{-1} : X \to X$. This is a consequence of the closed graph theorem and the fact that $D$ is closed.

Since both $i$ and $-i$ do not belong to the spectrum of $D$, the \emph{Cayley transform}
\[
C_D := (D - i)(D + i)^{-1} : X \to X
\]
defines a unitary operator with adjoint given by $C_D^* = (D + i)(D - i)^{-1} : X \to X$. Since $C_D \in \B L(X)$ is a unitary element we get that the spectrum of $C_D$, denoted by $\si(C_D)$, is contained in the unit circle $S^1 \su \B C$. Applying the identity
\[
C_D - 1 = -2i \cd (D + i)^{-1}
\]
we see that the following statements are equivalent:
\begin{enumerate}
\item $1$ belongs to the spectrum of $C_D$;
\item $\T{Dom}(D) \neq X$;
\item There is no bounded operator $T : X \to X$ such that $T(\xi) = D(\xi)$ for all $\xi \in \T{Dom}(D)$.
\end{enumerate}
Notice also that the map $c : \B C \sem \{-i\} \to \B C \sem \{1\}$ defined by $c(\la) := (\la - i)(\la + i)^{-1}$ is a homeomorphism with inverse given by $c^{-1}(\mu) = -i (\mu + 1)(\mu - 1)^{-1}$. The restriction of $c$ to $\B R$ yields a homeomorphism $c : \B R \to S^1 \sem \{1\}$.

The result of the next proposition depends on our assumption $X \neq \{0\}$.

\begin{prop}\label{p:cayley}
Let $\la \in \B C \sem \{-i\}$. It holds that $\la \in \si(D)$ if and only if $c(\la) \in \si(C_D)$. In particular, we get that $\si(D)$ is a closed and non-empty subset of $\B R$.
\end{prop}
\begin{proof}
  A couple of algebraic manipulations yield that
  \begin{equation}\label{eq:cayley}
  \begin{split}
    (c(\la) - 1) (\la - D)(D + i)^{-1}
    & = -2i (\la + i)^{-1} (\la - D)(D + i)^{-1} \\
    & = -2i (D + i)^{-1} + 2i (\la + i)^{-1} 
    = C_D - c(\la)
    \end{split}
  \end{equation}
  It can now be deduced from \eqref{eq:cayley} that $\la \in \si(D)$ if and only if $c(\la) \in \si(C_D)$.

  Since the homeomorphism $c : \B C \sem \{-i\} \to \B C \sem \{1\}$ restricts to a homeomorphism between $\B R$ and $S^1 \sem \{1\}$ and $\si(C_D) \su S^1$ is closed, we may conclude that $\si(D)$ is a closed subset of $\B R$. Suppose now for contradiction that $\si(D) = \emptyset$. Since $\si(C_D)$ is a non-empty subset of $S^1$ we then get that $\si(C_D) = \{1\}$ and because $C_D$ is normal we must have that $C_D = 1$. This entails that $0 = C_D - 1 = -2i (D + i)^{-1}$ which is impossible since $X \neq \{0\}$ and $(D + i)^{-1}$ has dense image in $X$. 
\end{proof}

The next result provides us with the relationship between our definition of the spectrum of $D$ and the definition from \cite{Wor:UAQ}, see \cite[Proposition 5.20]{Kus:FCR} for more details.

\begin{prop}
  It holds that $\la \in \B C \sem \si(D)$ if and only if there exists a bounded adjointable operator $R_\la : X \to X$ such that
  \[
R_\la (D - \la ) \su (D - \la ) R_\la = 1 .
\]
In this case, we have that $R_\la = (D - \la)^{-1}$ and $R_\la^* = (D - \ov{\la})^{-1}$.
\end{prop}
\begin{proof}
The only non-trivial issue is to show that if $\la \in \B C \sem \si(D)$, then $\ov{\la} \in \B C \sem \si(D)$ and the bounded operator $(D - \la)^{-1} : X \to X$ is adjointable with adjoint $(D - \ov{\la})^{-1}$. This follows since $\si(D) \su \B R$ by Proposition \ref{p:cayley}.
\end{proof}

We are now ready to present a version of the continuous functional calculus in terms of the Cayley transform. The notation $t : \si(D) \to \B C$ refers to the continuous function obtained from the inclusion $\si(D) \su \B C$. An application of Proposition \ref{p:cayley} ensures us that the image of $t$ is contained in $\B R \su \B C$.

\begin{thm}\label{t:func}
  There is a unique unital $*$-homomorphism $\Psi : C_b\big(\si(D) \big) \to \B L(X)$ satisfying that
  \[
\Psi\big( (t - i)(t + i)^{-1}\big) = C_D . 
\]
It moreover holds that $\Psi$ is injective and the restriction of $\Psi$ to $C_0\big(\si(D) \big)$ is non-degenerate.
\end{thm}
\begin{proof}
Using Proposition \ref{p:cayley} we may consider the unital $C^*$-subalgebra $C_{\pm \infty}\big( \si(D)\big) \su C_b\big( \si(D)\big)$ consisting of all continuous functions $f : \si(D) \to \B C$ satisfying that $f \ci c^{-1} : \si(C_D) \sem \{1\} \to \B C$ extends to a continuous function on $\si(C_D)$. We thus have that precomposition with $c^{-1}$ yields a $*$-isomorphism $C_{\pm \infty}\big(\si(D)\big) \cong C\big( \si(C_D)\big)$. Hence by applying the usual continuous functional calculus to the Cayley transform $C_D \in \B L(X)$ we get a unique unital $*$-homomorphism $\Psi : C_{\pm \infty}\big(\si(D) \big) \to \B L(X)$ such that $\Psi\big( (t - i)(t + i)^{-1}\big) = C_D$. Moreover, the unital $*$-homomorphism $\Psi$ is injective on $C_{\pm \infty}\big( \si(D) \big)$.
  
Consider now the $C^*$-subalgebra $C_0\big(\si(D) \big) \su C_{\pm \infty}\big(\si(D) \big)$ and notice that the continuous function $(t + i)^{-1}$ belongs to $C_0\big( \si(D) \big)$ and satisfies that $\Psi\big( (t + i)^{-1} \big) = (D + i)^{-1}$. Since $\T{Dom}(D) \su X$ is norm-dense and agrees with the image of $(D + i)^{-1} : X \to X$ we conclude that the restriction of $\Psi$ to $C_0\big( \si(D) \big)$ is non-degenerate. Consequently, we obtain from \cite[Proposition 2.1]{Lan:HCM} that $\Psi$ extends uniquely to a unital $*$-homomorphism $\Psi : C_b\big(\si(D) \big) \to \B L(X)$. Moreover, observing that the ideal $C_0\big( \si(D) \big)$ is essential in $C_b\big( \si(D) \big)$ we conclude that $\Psi$ is in fact injective on $C_b\big(\si(D)\big)$. 
\end{proof}

For an element $f \in C_b\big(\si(D) \big)$ we often apply the notation $f(D)$ instead of $\Psi(f)$. The next result presented here is a well-known consequence of Theorem \ref{t:func}. We give a few details on the proof. 

\begin{prop}\label{p:core}
  The right $A$-submodule
  \[
  \C X := \Psi\big( C_c(\si(D))\big)(X) := \T{span}_{\B C}\big\{ \Psi(f)(\xi) \mid f \in C_c\big(\si(D)\big) \,  , \, \,  \xi \in X \big\} \su X
  \]
  is a core for the selfadjoint and regular unbounded operator $D$ and for every $f \in C_c\big(\si(D)\big)$ we have that
  \[
  \Psi(f) D \su \Psi(t \cd f) = D \Psi(f) .
  \]
\end{prop}
\begin{proof}
  Let $f \in C_c\big(\si(D)\big)$. Since $\Psi\big( (t + i)^{-1}\big) = (D + i)^{-1}$ we get that $\Psi(f) = (D + i)^{-1}\cd \Psi\big( (t + i) \cd f \big)$. This entails that the image of $\Psi(f) : X \to X$ is contained in the domain of $D$ and that
  \[
  D \Psi(f) = D(D + i)^{-1} \Psi\big( (t + i) \cd f \big) = \Psi(t \cd f) .
    \]
  A similar argument shows that $\Psi(t \cd f)(\xi)$ agrees with $\Psi(f)(D \xi)$ for all $\xi \in \T{Dom}(D)$.
  
  It therefore only remains to show that $\C X$ is a core for $D : \T{Dom}(D) \to X$. Let $\xi \in \T{Dom}(D)$. Using the non-degeneracy statement in Theorem \ref{t:func} we may choose a sequence $\{ f_n\}$ in $C_c\big(\si(D) \big)$ such that for every $\eta \in X$ the sequence $\big\{ \Psi(f_n)(\eta) \big\}$ converges to $\eta$ with respect to the norm on $X$. But then we also know that the sequence $\big\{ D \Psi(f_n)(\xi) \big\} = \big\{ \Psi(f_n)(D \xi) \big\}$ converges to $D \xi$, entailing that $\C X$ is indeed a core for $D$.
\end{proof}

An application of Proposition \ref{p:core} shows that the functional calculus obtained in Theorem \ref{t:func} agrees with the functional calculus described in \cite[Theorem 1.6]{Wor:UAQ} and \cite[Theorem 3.4]{Kus:FCR} (specialized to the selfadjoint and regular setting).   
\medskip

Given a real valued (possibly unbounded) continuous function $g: \si(D) \to \B R$ we may apply the core $\C X$ from Proposition \ref{p:core} to define a selfadjoint and regular unbounded operator $g(D) : \T{Dom}\big(g(D)\big) \to X$. The selfadjoint and regular unbounded operator $g(D)$ has $\C X$ as a core and for every $f \in C_c\big(\si(D) \big)$ and $\xi \in X$ we have the explicit formula
\[
g(D) \cd \Psi(f)(\xi) := \Psi(g \cd f)(\xi) .
\]
For more details on this construction we refer to \cite[Proposition 10.7]{Lan:HCM}. The spectrum of $g(D)$ agrees with the closure of the image of $g : \si(D) \to \B R$, yielding the formula: 
\begin{equation}\label{eq:spectrum}
\si\big( g(D) \big) = \ov{ g\big( \si(D) \big) },
\end{equation}
see \cite[Result 3.10]{Kus:FCR}. The transformation rule stated in the next proposition is contained in \cite[Proposition 3.11]{Kus:FCR}. 

\begin{prop}\label{p:trans}
  Let $f : \si(D) \to \B R$ be a continuous function and let $g : \si\big(f(D)\big) \to \B R$ be another continous function. We have the identity
  \[
(g \ci f)(D) = g\big( f(D) \big) .
  \]
\end{prop}

Using the continuous functional calculus for $D$ we may define the \emph{bounded transform} $F_D = \Psi\big(t(1 + t^2)^{-1/2} \big)$ which is a selfadjoint element in $\B L(X)$ with $\| F_D \|_\infty \leq 1$. It can in fact be verified that the image of
$(1 + D^2)^{-1/2} = \big( R_i^* R_i \big)^{1/2}$ is contained in the domain of $D$, yielding the alternative description of the bounded transform $F_D = D(1 + D^2)^{-1/2}$.


\subsection{The Fourier transform}\label{ss:fourier}
In this subsection we continue in the setting described in the beginning of Section \ref{s:func}. For a particular class of bounded continuous functions on the real line we describe the associated functional calculus in terms of the Fourier transform. It is in this respect convenient for us to work with Bochner integrals and we refer to \cite{HiPh:FAS} for an account of this theory. Throughout this subsection, the restriction map from $C_b(\B R)$ to $C_b\big(\si(D)\big)$ is suppressed. 

As a consequence of Theorem \ref{t:func} we have the following lemma. Notice that the strict continuity mentioned in the statement means that the map from $\B R$ to $X$ given by $s \mapsto e^{is D}(\xi)$ is continuous for all $\xi \in X$. 

\begin{lemma}\label{l:struni}
The family $\{ e^{is D} \}_{s \in \B R}$ forms a strictly continuous group of unitary operators on $X$ and if $\xi \in \T{Dom}(D)$, then the map $s \mapsto e^{is D}(\xi)$ is continuously differentiable with derivative $s \mapsto e^{is D} i D(\xi)$. Moreover, if $f \in C_0(\B R)$, then the map $s \mapsto e^{is D} f(D)$ is continuous with respect to the operator norm on $\B L(X)$.
\end{lemma}

The integrals appearing in the statement and the proof of the next proposition are Bochner integrals. 

\begin{prop}\label{p:fourier}
  Let $f : \B R \to \B C$ be a bounded continuous $L^1$-function. If the Fourier transform
  $\widehat{f} : \B R \to \B C$ is an $L^1$-function, then for every $\xi \in X$ the function $p \mapsto \widehat{f}(p) e^{i p D}(\xi)$ is Bochner integrable (as a map from $\B R$ to $X$) and
\begin{equation}\label{eq:fourier}
f(D)(\xi) = \frac{1}{\sqrt{2\pi}} \int_{-\infty}^\infty \widehat{f}(p) e^{i p D}(\xi) \, dp . 
\end{equation}
\end{prop}
\begin{proof}
  Let $\xi \in X$. By the Riemann-Lebesgue lemma we know that the Fourier transform $\widehat{f} : \B R \to \B C$ is a continuous function vanishing at infinity and the integrand of \eqref{eq:fourier} is therefore continuous by Lemma \ref{l:struni}. Since $\widehat{f}$ is assumed to be an $L^1$-function we now get that the integrand is Bochner integrable as a map from $\B R$ to $X$. In fact, it follows that the operator on $X$ determined by the right hand side of \eqref{eq:fourier} is bounded adjointable with operator norm dominated by the Lebesgue integral $\frac{1}{\sqrt{2\pi}} \cd \int_{-\infty}^\infty |\widehat{f}(p)| \, dp$.

  Let $g \in C_0(\B R)$. An argument similar to the one just presented shows that the continuous from $\B R$ to $C_0(\B R)$ given by $p \mapsto \widehat{f}(p) e^{i p t} \cd g$ is also Bochner integrable. Moreover, upon using the inverse Fourier transform, we get that 
\begin{equation}\label{eq:product}
f \cd g = \frac{1}{\sqrt{2\pi}} \int_{-\infty}^\infty \widehat{f}(p) e^{i p t} \cd g  \, dp . 
\end{equation}
Applying the functional calculus $\Psi : C_b(\B R) \to \B L(X)$ to both sides of \eqref{eq:product} yields that
\begin{equation}\label{eq:fourier2}
f(D) g(D) = \frac{1}{\sqrt{2\pi}} \int_{-\infty}^\infty \widehat{f}(p) e^{i p D} g(D)  \, dp  . 
\end{equation}

The result of the proposition now follows by using that the restriction of $\Psi$ to $C_0(\B R)$ is non-degenerate, see Theorem \ref{t:func}.
\end{proof}

\subsection{Commutator estimates}\label{ss:commutator}
We remain in the setting described in the beginning of Section \ref{s:func}. Our aim is now to estimate commutators between certain bounded adjointable operators arising from the continuous functional calculus and Lipschitz operators in the sense of the following definition. As in the previous subsection, the various integrals appearing in this subsection are Bochner integrals.

\begin{dfn}\label{d:lipschitz}
  We say that a bounded adjointable operator $T : X \to X$ is a \emph{Lipschitz operator} with respect to $D$, if
 \begin{enumerate}
 \item $T\big( \T{Dom}(D) \big) \su \T{Dom}(D)$ and;
 \item The commutator $[D,T] = D T - T D : \T{Dom}(D) \to X$ extends to a bounded adjointable operator $d(T) : X \to X$.
 \end{enumerate}
 The set of Lipschitz operators with respect to $D$ is denoted by $\T{Lip}_D(X)$.
\end{dfn}

It can be verified that $\T{Lip}_D(X) \su \B L(X)$ is a unital $*$-subalgebra and the assignment $T \mapsto d(T)$ is a closed $*$-derivation, meaning that $d : \T{Lip}_D(X) \to \B L(X)$ is closed and satisfies that
\[
d(S T) = d(S) T + S d(T) \, \, \T{ and } \, \, \, d(T^*) = - d(T)^*
\]
for all $S,T \in \T{Lip}_D(X)$. We turn $\T{Lip}_D(X)$ into a unital Banach $*$-algebra with respect to the norm $\| T \|_1 := \| T \|_\infty + \| d(T) \|_\infty$ and record that the inclusion $\T{Lip}_D(X) \to \B L(X)$ is a contraction. For every $n \in \B N$, the closed $*$-derivation $d : \T{Lip}_D(X) \to \B L(X)$ yields a closed $*$-derivation
\begin{equation}\label{eq:derivation}
d : M_n\big( \T{Lip}_D(X)\big) \to \B L(X^{\op n})
\end{equation}
obtained by applying $d$ entry by entry and identifying $M_n\big(\B L(X)\big)$ with $\B L(X^{\op n})$.

The following lemma is a version of Duhamel's formula.

\begin{lemma}\label{l:duhamel}
  Let $T \in \T{Lip}_D(X)$. It holds that
  \[
[e^{i D},T](\xi) = i \int_0^1 e^{s  i  D} d(T) e^{(1-s) i D}(\xi) \, ds \q \mbox{for all } \xi \in X .
  \]
\end{lemma}
\begin{proof}
  By density of $\T{Dom}(D) \su X$ it suffices to consider a vector $\xi \in \T{Dom}(D)$. Since $T(\xi) \in \T{Dom}(D)$ we get from Lemma \ref{l:struni} that the map from $\B R$ to $X$ given by $s \mapsto e^{s i D} T e^{(1 - s) iD}(\xi)$ is continuously differentiable with derivative
  \[
  i e^{s i D} D T e^{ (1 - s)  iD}(\xi) - i  e^{s i D} T e^{(1 - s) iD} D(\xi)
  = i e^{s i D} d(T) e^{ (1 - s) iD}(\xi) .
  \]
  The result of the lemma now follows from the fundamental theorem of calculus.
\end{proof}

The next proposition is related to \cite[Theorem 3.2.32]{BrRo:OAQ}. 

\begin{prop}\label{p:estimate}
  Let $f : \B R \to \B C$ be a bounded continuous $L^1$-function satisfying that $p \cd \widehat{f}(p)$ is an $L^1$-function. We have the operator norm estimate
  \[
  \big\| [ f(D),T] \big\|_\infty \leq \| d(T) \|_\infty \cd \frac{1}{\sqrt{2\pi}} \cd \| p \cd \widehat{f}(p) \|_1 
  \q \mbox{for all } T \in \T{Lip}_D(X) .
\]
\end{prop}
\begin{proof}
  Let $T \in \T{Lip}_D(X)$ and $\xi \in X$ be given. An application of Proposition \ref{p:fourier} and Lemma \ref{l:duhamel} entail that
  \[
\begin{split}
    [f(D),T](\xi) & = \frac{1}{\sqrt{2 \pi}} \int_\infty^\infty \widehat{f}(p) [ e^{ipD},T](\xi) \, dp \\
    & = \frac{1}{\sqrt{2 \pi}} \int_\infty^\infty \widehat{f}(p) i p \cd
    \left( \int_0^1 e^{s i p D} d(T) e^{(1 - s) ip D}(\xi) \, ds \right) \, dp
\end{split}
\]
and hence that
\[
\big\| [f(D),T](\xi) \big\|_X \leq \frac{1}{\sqrt{2 \pi}} \int_\infty^\infty \big| p \cd \widehat{f}(p) \big| \, dp \cd \| d(T)(\xi) \|_X .
\]
This proves the proposition.
\end{proof}

The following corollary plays a fundamental role in the present text.

\begin{cor}\label{c:scaleest}
 Let $f : \B R \to \B C$ be a bounded continuous $L^1$-function satisfying that $p \cd \widehat{f}(p)$ is an $L^1$-function.  For every $\rho > 0$ and $T \in \T{Lip}_D(X)$ we have the operator norm estimate
  \[
  \big\| [ f(D/\rho),T] \big\|_\infty \leq \frac{1}{\rho \sqrt{2\pi}} \cd \| d(T) \|_\infty \cd \| p \cd \widehat{f}(p) \|_1 .
\]
\end{cor}
\begin{proof}
  Let $\rho > 0$ be given. Define the bounded continuous $L^1$-function $f_\rho(x) := f(x/\rho)$ and remark that $p \cd \widehat{f_\rho}(p) = \rho \cd p \cd \widehat{f}(\rho \cd p)$ and hence that
  \[
\| p \cd \widehat{f_\rho}(p) \|_1 = \frac{1}{\rho} \cd \| p \cd \widehat{f}(p) \|_1 .
  \]
The result of the present corollary is now a consequence of Proposition \ref{p:estimate}. 
\end{proof}

\subsection{Perturbations}\label{ss:perturb}
We still consider the setting described in the beginning of Section \ref{s:func}. The purpose of this subsection is to study the behaviour of the functional calculus under perturbations. On top of the selfadjoint and regular unbounded operator $D : \T{Dom}(D) \to X$ we fix a selfadjoint bounded operator $R : X \to X$. Remark that the Kato-Rellich theorem for Hilbert $C^*$-modules implies that $D + R : \T{Dom}(D) \to X$ is again selfadjoint and regular, see for example \cite[Theorem 4.5]{KaLe:LGR}.

\begin{prop}\label{p:perturb}
  Let $f : \B R \to \B C$ be a bounded continuous $L^1$-function satisfying that $p \cd \widehat{f}(p)$ is an $L^1$-function. We have the operator norm estimate
  \[
\| f(D + R) - f(D) \|_\infty \leq \frac{1}{\sqrt{2\pi}} \cd \| R \|_\infty \cd \| p \cd \widehat{f}(p) \|_1 .
  \]
\end{prop}
\begin{proof}
  Define the selfadjoint and regular unbounded operator $D \op (D+ R) : \T{Dom}(D)^{\op 2} \to X^{\op 2}$ together with the bounded adjointable operator $T := \ma{cc}{0 & 0 \\ 1 & 0} : X^{\op 2} \to X^{\op 2}$. It holds that $T$ is a Lipschitz operator with respect to $D \op (D + R)$ and the commutator $\big[ D \op (D+ R), T \big] : \T{Dom}(D)^{\op 2} \to X^{\op 2}$ extends to the bounded adjointable operator $\ma{cc}{0 & 0 \\ R & 0}$. Moreover, we have that $f\big(D \op (D + R) \big) = f(D) \op f(D + R)$ and hence that
  \[
\big[ f\big(D \op (D + R) \big), T \big] = \ma{cc}{0 & 0 \\ f(D + R) - f(D) & 0} .
\]
These observations show that the result of the current proposition follows from Proposition \ref{p:estimate}. 
\end{proof}

\section{$K$-theory}\label{s:ktheory}
Let $A$ be a $C^*$-algebra and recall that $A^\sim$ denotes the unitalization of $A$.

We start this section by reviewing some standard results concerning $K$-theory and stability under holomorphic functional calculus. Let us fix a norm-dense $*$-subalgebra $\C A \su A$ satisfying the following conditions:

  \begin{enumerate}
  \item $\C A$ is equipped with a norm $\| \cd \|_1 : \C A \to [0,\infty)$ turning $\C A$ into a Banach $*$-algebra and the inclusion $\io : \C A \to A$ is continuous; 
  \item Letting $\C A^\sim$ denote the unitalization of the Banach $*$-algebra $\C A$, the inclusion $\io : \C A^\sim \to A^\sim$ (given by $\io(a,\la) := (\io(a),\la)$) is assumed to be \emph{spectrally invariant}, thus for every element $(a,\la) \in \C A^\sim$ the spectrum of $(a,\la)$ agrees with the spectrum of $(\io(a),\la) \in A^\sim$. 
  \end{enumerate}

  We turn $\C A^\sim$ into a unital Banach $*$-algebra with respect to the norm $\| (a,\la) \|_1 := \| a \|_1 + |\la|$ and for every $n \in \B N$ we equip $M_n(\C A^\sim)$ with the structure of a unital Banach $*$-algebra satisfying that the map
  \[
\C A^\sim \to M_n(\C A^\sim) \q (a,\la) \mapsto (a,\la) e_{ij}
\]
is an isometry for all $i,j \in \{1,\ldots,n\}$ (where $e_{ij}$ is the matrix with $1$ in position $(i,j)$ and zeroes elsewhere). Let us clarify that the inclusion $\io : M_n(\C A^\sim) \to M_n(A^\sim)$ is automatically spectrally invariant and we refer to \cite[Lemma 2.1]{Swa:TEP} for a proof of this fact, but see also \cite[Proposition A.2.2]{Bos:OKA}.

The next result is a straightforward consequence of properties of the holomorphic functional calculus, see \cite[Chapitre I, \S 4, Th\'eor\`eme 3 and Proposition 7]{Bou:TS12}.

\begin{prop}\label{p:holomorphic}
Let $n \in \B N$ and let $x \in M_n(\C A^\sim)$. If $f$ is a holomorphic function defined on an open subset of the spectrum of $\io(x) \in M_n(A^\sim)$, then it holds that $f(\io(x))$ belongs to $M_n(\C A^\sim)$ and we have the identity $f(\io(x)) = \io( f(x))$. In particular, we get that $\C A^\sim$ is a local $C^*$-algebra inside $A^\sim$ in the sense of \cite[Definition 3.1.1]{Bla:KOA}.
\end{prop}

From a $K$-theoretic perspective it is impossible to distinguish $\C A$ from $A$, see for example \cite[Appendix 3]{Con:TIC}, \cite[Th\'eor\`eme A.2.1]{Bos:OKA} or \cite[Chapter II]{Bla:KOA} for a proof:

\begin{prop}\label{p:ktheory}
The $K$-theory of the Banach $*$-algebra $\C A$ is isomorphic to the $K$-theory of the $C^*$-algebra $A$ via the group homomorphism induced by the inclusion $\io : \C A \to A$.
\end{prop}

\subsection{Relative $K$-theory}\label{ss:relative}
In this paper, we focus less on the usual picture of even $K$-theory in terms of projections, see e.g. \cite{RLL:IKC}. Instead, we work with selfadjoint invertibles following, to some extent, the approach developed in \cite{GrSc:IPS,BoLo:KRU}. In order to get a sufficient amount of flexibility we let $A$ be a $C^*$-algebra and consider a unital $C^*$-algebra $L$ which contains $A$ as a norm-closed $*$-ideal. One possible choice could be $L = A^\sim$ but there are many other possibilities, for example we could consider $L = \B L(A)$ where $A$ is considered as a Hilbert $C^*$-module over itself. In this setting $A$ can be identified with the compact operators $\B K(A)$ sitting as a norm-closed $*$-ideal inside $\B L(A)$.

We define the unital $C^*$-subalgebra $D(L,A) \su L \op L$ by putting
\[
D(L,A) := \big\{ (x,y) \in L \op L \mid x - y \in A \big\} 
\]
and record that $D(L,A)$ comes equipped with two unital $*$-homomorphisms $\rho_1$ and $\rho_2 : D(L,A) \to L$ given by $\rho_1(x,y) = x$ and $\rho_2(x,y) = y$. The relative $K$-theory group $K_0(L,A)$ is then defined as the kernel of the group homorphism induced by $\rho_1$, meaning that
\[
K_0(L,A) := \T{Ker}\Big( K_0(\rho_1) : K_0\big( D(L,A) \big) \to K_0(L) \Big) . 
\]

We also have two $*$-homomorphisms $i_1$ and $i_2 : A \to D(L,A)$ satisfying that $i_1(x) = (x,0)$ and $i_2(x) = (0,x)$. In our context the excision theorem, see e.g. \cite[Theorem 1.5.9]{Ros:AKA}, can be formulated as follows:

\begin{thm}\label{t:excision}
The $*$-homomorphism $i_2 : A \to D(L,A)$ induces an isomorphism of abelian groups
\[ 
K_0(i_2) : K_0(A) \to K_0(L,A) .
\] 
\end{thm}

For each $n \in \B N$, let $\T{GL}_n\big(D(L,A)\big)$ denote the group of invertible elements in the unital $C^*$-algebra $M_n\big( D(L,A) \big)$. Let us introduce the set
\[
W_n(L,A) := \big\{ H \in \T{GL}_n\big(D(L,A)\big) \mid H = H^* \big\} 
\]
consisting of all the selfadjoint invertible elements in $M_n\big( D(L,A) \big)$. Notice that the unit $E_n := (1,1)^{\op n}$ in $M_n\big( D(L,A)\big)$ belongs to $W_n(L,A)$. Moreover, letting $\De : L \to D(L,A)$ denote the unital $*$-homomorphism given by $\De(x) := (x,x)$ we get an element $\De(G) \in W_n(L,A)$ whenever $G$ is an invertible selfadjoint element in $M_n(L)$.

To ease the notation, for every element $H \in W_n(L,A)$ we define the mutually orthogonal projections
\begin{equation}\label{eq:mutuproj}
H_+ := \frac{ E_n + H \cd |H|^{-1}}{2} \, \, \T{ and } \, \, \, H_- := \frac{ E_n - H \cd |H|^{-1}}{2} 
\end{equation}
inside the unital $C^*$-algebra $M_n\big(D(L,A) \big)$. It holds that $H_+ + H_- = E_n$ whereas $H_+ - H_- = H \cd |H|^{-1}$. 

We endow $W_n(L,A)$ with the metric topology inherited from the $C^*$-norm on $M_n\big( D(L,A) \big)$.

\begin{dfn}\label{d:homotopy}
Let $n \in \B N$. Two elements $H_0$ and $H_1$ in $W_n(L,A)$ are said to be \emph{homotopic}, if there exists a continuous map $\al : [0,1] \to W_n(L,A)$ such that $\al(0) = H_0$ and $\al(1) = H_1$.
\end{dfn}

We define the set $W(L,A)$ as the disjoint union
\[
W(L,A) := \bigsqcup_{n \in \B N} W_n(L,A) 
\]
and equip $W(L,A)$ with the additive operation ``$\op$'' induced by the direct sum of matrices with entries in $D(L,A)$.


\begin{dfn}\label{d:stablehomo}
  Let $n,m \in \B N$. Two elements $H_0 \in W_n(L,A)$ and $H_1 \in W_m(L,A)$ are called \emph{stably homotopic}, if there exists $k \in \B N$ with $k \geq n,m$ and selfadjoint invertible elements $G_0 \in M_{k - n}(L)$ and $G_1 \in M_{k - m}(L)$ such that $H_0 \op \De(G_0)$ and $H_1 \op \De(G_1)$ are homotopic as elements in $W_k(L,A)$. In this case, we write $H_0 \sim H_1$. In the special situation we
\end{dfn}

It is not difficult to verify that the relation of being stably homotopic is an equivalence relation on $W(L,A)$ which is compatible with the additive operation. We notice in passing that the special case where $k = n = m$ in Definition \ref{d:stablehomo} just amounts to saying that $H_0$ and $H_1$ are homotopic in the sense of Definition \ref{d:homotopy}. 

For future reference, we record the following:

\begin{lemma}\label{l:phase}
Let $n\in \B N$ and let $H \in W_n(L,A)$. It holds that $H$ is homotopic to the selfadjoint unitary $H \cd |H|^{-1}$ via the continuous map $t \mapsto H \cd |H|^{-t}$.
\end{lemma}

\begin{lemma}\label{l:abelian}
The quotient space $W(L,A)/ \sim$ is an abelian group with neutral element given by the equivalence class of $E_1 = (1,1) \in W(L,A)$. The inverse of an equivalence class $[H] \in W(L, A)/\sim$ agrees with $[-H]$.
\end{lemma}
\begin{proof}
  Let $H \in W_n(L,A)$ for some $n \in \B N$ and let us focus on showing that $H \op (-H)$ is homotopic to $E_n \op (- E_n)$ where we recall that $E_n = (1,1)^{\op n}$ denotes the unit in $M_n\big( D(L,A) \big)$. This entails that $[H] + [-H] = [E_1]$ inside the quotient space $W(L,A)/\sim$. Using Lemma \ref{l:phase} we may assume, without loss of generality that, $H$ is a selfadjoint unitary element in $M_n\big( D(L,A) \big)$. Applying the notation from \eqref{eq:mutuproj} we then have that $H \op (-H) = (H_+ - H_-) \op (H_- - H_+)$ and $E_n \op (- E_n) = (H_+ + H_-) \op (-H_+ - H_-)$. Upon defining the unitary matrices
\[
U_t := \ma{cc}{ \cos(\pi t/2) H_- + H_+ & -\sin(\pi t/2) H_- \\ \sin(\pi t/2) H_- & \cos(\pi t/2) H_- + H_+} \q t \in [0,1] ,
\]
we get the homotopy between $H \op (-H)$ and $E_n \op (- E_n)$ via the continuous map $\al : [0,1] \to W_{2n}(L,A)$ given by the formula
\[
\al(t) := U_t \cd \big(H \op (-H)\big) \cd U_t^* . \qedhere
\]
\end{proof}

We apply the notation
\[
K_0^{\T{inv}}(L,A) := W(L,A)/\sim
\]
for the abelian group appearing in Lemma \ref{l:abelian}.

%

\begin{thm}\label{t:inviso}
  The abelian group $K_0^{\T{inv}}(L,A)$ is isomorphic to the relative $K$-group $K_0(L,A)$. For $H \in W_n(L,A)$, the relevant isomorphism $\varphi$ sends the equivalence class $[H]$ to the formal difference of projections
  \begin{equation}\label{eq:defphi}
\varphi( [H] ) := \big[ H_+ \big] -  \big[ (\De \rho_1)(H_+) \big] .
\end{equation}
  For projections $p \in M_n( D(L,A))$ and $q \in M_m(D(L,A))$ satisfying that $[p] - [q] \in K_0(L,A)$, the inverse of $\varphi$ is given explicitly by the formula
  \[
\varphi^{-1}( [p] - [q] ) := \big[ (2 p - E_n) \op (E_m - 2 q )  \big] .
  \]
\end{thm}
\begin{proof}
  Let us first verify that both $\varphi$ and $\varphi^{-1}$ are well-defined. 

  To show that $\varphi$ is well-defined, remark first of all that the left hand side of \eqref{eq:defphi} does indeed belong to $K_0(L,A) \su K_0\big( D(L,A) \big)$ since $\rho_1 \De \rho_1 = \rho_1$. Suppose now that $H_0 \in W_n(L, A)$ and $H_1 \in W_m(L, A)$ are stably homotopic. Choose $k \geq n,m$ together with selfadjoint invertibles $G_0 \in M_{k-n}(L)$ and $G_1 \in M_{k-m}(L)$ such that $H_0 \op \De(G_0)$ and $H_1 \op \De(G_1)$ are homotopic. This entails that $(H_0)_+ \op \De(G_0)_+$ and $(H_1)_+ \op \De(G_1)_+$ are homotopic as projections in $M_k\big( D(L,A) \big)$ and we therefore have the identities 
  \[
  \begin{split}
  \big[ (H_0)_+ \big] - \big[ (\De \rho_1)(H_0)_+ \big] & =  \big[ (H_0)_+ \op \De(G_0)_+ \big] - \big[ (\De \rho_1)(H_0)_+ \op \De(G_0)_+ \big] \\
  & = \big[ (H_1)_+ \op \De(G_1)_+ \big] - \big[ (\De \rho_1)(H_1)_+ \op \De(G_1)_+ \big] \\
  & = \big[ (H_1)_+ \big] - \big[ (\De \rho_1)(H_1)_+ \big]
  \end{split}
 \]
 inside $K_0( L,A) \su K_0\big( D(L,A) \big)$. Notice in this respect that $\De \rho_1 \De = \De$. This shows that $\varphi$ is a well-defined group homomorphism (the compatibility with the additive structures being obvious). 

 To show that $\varphi^{-1}$ is well-defined let $p_0 \in M_n\big( D(L,A) \big)$ and $p_1 \in M_m\big( D(L,A) \big)$ be two projections and suppose that there exists a $k \geq n,m$ such that the stabilizations $p_0 \op 0_{k - n}$ and $p_1 \op 0_{k - m}$ are homotopic as projections inside $M_k\big( D(L,A) \big)$. It then holds that the elements $2 p_0 - E_n$ and $2 p_1 - E_m$ are stably homotopic as elements in $W(L,A)$, stabilizing with $-E_{k-n} = \De(-1_{k-n})$ and $-E_{k-m} = \De(-1_{k-m})$, respectively. We therefore get a well-defined group homomorphism from $K_0\big( D(L,A) \big)$ to $K_0^{\T{inv}}(L,A)$ and $\varphi^{-1}$ is the restriction of this group homomorphism to $K_0(L,A) \su K_0\big( D(L,A) \big)$.

 It remains to establish that $\varphi$ and $\varphi^{-1}$ are indeed each others inverses. To show that $\varphi^{-1} \varphi$ is equal to the identity we let $H \in W_n(L,A)$ and compute that
 \[
 \begin{split}
   (\varphi^{-1} \varphi)( [ H ] )
   & = \big[ (2 H_+ - E_n) \op ( E_n - 2 (\De \rho_1)(H_+) ) \big] = [ 2 H_+ - E_n ] \\
   & = \big[ H \cd |H|^{-1} \big] = [H] ,
 \end{split}
 \]
 where the last identity follows from Lemma \ref{l:phase}. On the other hand, to see that $\varphi \varphi^{-1}$ is equal to the identity we let $p \in M_n( D(L,A))$ and $q \in M_m(D(L,A))$ be projections such that $\big[ \rho_1(p) \big] = \big[ \rho_1(q) \big]$ inside $K_0(L)$. It then holds that
 \[
 \begin{split}
 ( \varphi \varphi^{-1} )( [p] - [q] ) & = \big[ p \op (E_m - q) \big] - \big[ (\De \rho_1)(p) \op (\De \rho_1)(E_m - q) \big] \\
 & = [p] + [E_m -q] - \big[ (\De \rho_1)(p) \big] - \big[ E_m - (\De \rho_1)(q) \big] \\
 & = [p] - [q] + \big[ (\De \rho_1)(q) \big] - \big[ (\De \rho_1)(p) \big] = [p] - [q] . \qedhere
 \end{split}
 \]
\end{proof}

\subsection{Functoriality}
We continue by investigating the functoriality properties of our $K$-groups. Let us consider two $C^*$-algebras $A$ and $B$ and fix unital $C^*$-algebras $L$ and $N$ such that $A$ sits as a norm-closed $*$-ideal inside $L$ and $B$ sits as a norm-closed $*$-ideal inside $N$.

Suppose that $\si : L \to N$ is a $*$-homomorphism (which is not necessarily unital) satisfying that $\si(a) \in B$ for all $a \in A$. We then get an induced $*$-homomorphism $\si : D(L,A) \to D(N,B)$ given by $\si(x,y) := \big( \si(x) , \si(y) \big)$ which in turn induces a group homomorphism
\[
K_0(\si) : K_0(L,A) \to K_0(N,B) \q K_0(\si)( [p] - [q]) = [\si(p)] - [\si(q)] 
\]
at the level of relative $K$-theory groups. Similarly, using the functoriality of $K$-theory we get a group homomorphism $K_0(\si) : K_0(A) \to K_0(B)$ which only depends on the restriction of $\si : L \to N$ to the closed $*$-ideal $A \su L$. For every $n \in \B N$, we also have an induced map
\[
W_n(\si) : W_n(L,A) \to W_n(N,B) \q W_n(\si)(H) := \si(H) + E_n - \si(E_n)  
\]
and these maps can be gathered into a group homomorphism
\[
K_0^{\T{inv}}(\si) : K_0^{\T{inv}}(L,A) \to K_0^{\T{inv}}(N,B) .
\]

The relationship between our various group homomorphisms is explained in the next proposition:

\begin{prop}\label{p:functoriality}
The diagram here below is commutative:
\[
\begin{CD}
  K_0(A) @>{K_0(i_2)}>> K_0(L,A) @<{\varphi}<< K_0^{\T{inv}}(L,A) \\ 
  @V{K_0(\si)}VV @V{K_0(\si)}VV @V{K_0^{\T{inv}}(\si)}VV \\
  K_0(B) @>{K_0(i_2)}>> K_0(N,B) @<{\varphi}<< K_0^{\T{inv}}(N,B) 
\end{CD}
\]
\end{prop}
\begin{proof}
  The commutativity of the leftmost diagram follows immediately by functoriality of even $K$-theory so we focus on the commutativity of the rightmost diagram. Let thus $H \in W_n\big(D(L,A)\big)$ for some $n \in \B N$ because of Lemma \ref{l:phase} we may assume without loss of generality that $H$ is a selfadjoint unitary. Remark that $\si(H) + E_n - \si(E_n)$ is a selfadjoint unitary satisfying that
  \[
\frac{E_n + \si(H) + E_n - \si(E_n)}{2} = \si(H_+) + E_n - \si(E_n) .
  \]
The result of the lemma now follows from the computation
  \[
  \begin{split}
 (\varphi K_0^{\T{inv}}(\si) \big)( [H]) & = \varphi\big( \big[ \si(H) + E_n - \si(E_n) \big] \big) \\
  & = [ \si(H_+) + E_n - \si(E_n) ] - \big[ (\De \rho_1\si)(H_+) + E_n - \si(E_n) \big] \\
  & = \big[\si(H_+)] - [ (\si \De \rho_1)(H_+)\big]
  = \big( K_0(\si) \varphi \big)( [H] )  . \qedhere
  \end{split}
  \]
\end{proof}
 
%

\subsection{Half signatures and isomorphism with the integers}
In this subsection we investigate the isomorphism between $K_0^{\T{inv}}(A,A)$ and $K_0(A)$ in the case where $A$ is a unital $C^*$-algebra. We then apply this result to obtain an explicit isomorphism between $K_0^{\T{inv}}( \B C, \B C)$ and the integers in terms of half signatures. This does, at least to some extent, explain the appearance of half signatures in \cite{LoSc:SLE,LoSc:FVC}.  

For a unital $C^*$-algebra $A$ and a selfadjoint invertible matrix $H \in M_n(A)$ we recall the definition of the spectral projections $H_+$ and $H_-$ from \eqref{eq:mutuproj}. 

\begin{prop}\label{p:isomunital}
  Let $A$ be a unital $C^*$-algebra. The map which sends a selfadjoint invertible element $(H,G) \in D( M_n(A), M_n(A) )$ to the class $[G_+] - [H_+] \in K_0(A)$ induces an isomorphism of $K$-groups $K_0^{\T{inv}}(A,A) \cong K_0(A)$. The inverse is induced by the map which sends a projection $p \in M_n(A)$ to the class $\big[ (-1_n, 2p - 1_n) \big] \in K_0^{\T{inv}}(A,A)$.
\end{prop}
\begin{proof}
  We already know from Theorem \ref{t:excision} and Theorem \ref{t:inviso} that $K_0(A)$ is isomorphic to $K_0^{\T{inv}}(A,A)$ via the group isomorphism induced by the map $p \mapsto \big[ (-1_n, 2p - 1_n) \big]$ for $n \in \B N$ and projections $p \in M_n(A)$. Hence, given a selfadjoint invertible element $(H,G) \in D(M_n(A),M_n(A))$ it suffices to show that $\big[ (H,G) \big]$ agrees with $\big[ (-1_n, 2 G_+ - 1_n) \op (1_n, 1_n - 2H_+)\big]$ inside the $K$-group $K_0^{\T{inv}}(A,A)$. However, by Lemma \ref{l:phase} we may assume without loss of generality that $H$ and $G$ are both selfadjoint unitaries in $M_n(A)$. Remark now that the argument provided in the proof of Lemma \ref{l:abelian} shows that $H \op (-H)$ is homotopic to $(-1_n) \op (1_n)$ via a continuous path of selfadjoint unitaries in $M_{2n}(A)$. This implies that
  \[
\big[ (H,G) \big] = \big[ \big(H \op (-H), G \op (-H) \big)\big] = \big[ (-1_n, G) \op (1_n, - H) \big] 
\]
inside the $K$-group $K_0^{\T{inv}}(A,A)$. Since $G = 2 G_+ - 1_n$ and $-H = 1_n - 2 H_+$ we have proved the present proposition. 
\end{proof}

Recall that the signature of a selfadjoint invertible matrix $H \in M_m(\B C)$, denoted by $\T{sign}(H) \in \B Z$, is given by the number of positive eigenvalues minus the number of negative eigenvalues. It therefore holds that
\[
\T{sign}(H) = \T{TR}\big( H |H|^{-1} \big) = \T{TR}(H_+) - \T{TR}(H_-) ,
\]
where $\T{TR} : M_m(\B C) \to \B C$ refers to the (non-normalized) trace. In the statement and proof of the next proposition we are tacitly identifying $M_n\big(M_k(\B C)\big)$ with $M_{nk}(\B C)$ via the isomorphism which forgets the subdivisions.

\begin{prop}\label{p:signature}
  Let $k \in \B N$. The map which sends a selfadjoint invertible element $(H,G) \in D\big( M_n(M_k(\B C)), M_n( M_k(\B C)) \big)$ to the integer
  \[
\frac{1}{2} \big( \T{sign}(G) - \T{sign}(H) \big)
\]
induces an isomorphism of abelian groups $K_0^{\T{inv}}\big( M_k(\B C), M_k(\B C) \big) \cong \B Z$.
\end{prop}
\begin{proof}
  Recall that the isomorphism between $K_0\big( M_k(\B C)\big)$ and the integers is induced by the trace, so that a formal difference of projections $[p] - [q] \in K_0\big( M_k(\B C)\big)$ is mapped to $\T{TR}(p) - \T{TR}(q)$, see \cite[Example 3.3.2]{RLL:IKC}. An application of Proposition \ref{p:isomunital} now shows that $K_0^{\T{inv}}\big( M_k(\B C), M_k(\B C) \big)$ is isomorphic to $\B Z$ via the map which sends a selfadjoint invertible element $(H,G) \in  D\big( M_n(M_k(\B C)), M_n( M_k(\B C)) \big)$ to the integer $\T{TR}(G_+) - \T{TR}(H_+)$. The result of the proposition therefore follows by noting that
  \[
  \begin{split}
  \frac{1}{2}\big( \T{sign}(G) - \T{sign}(H) \big)
  & = \frac{1}{2}\big(  \T{TR}( 2 G_+ - 1_{nk}) - \T{TR}(2 H_+ - 1_{nk}) \big) \\
  & = \T{TR}(G_+) - \T{TR}(H_+) . \qedhere
  \end{split}
  \]
\end{proof}








\section{$KK$-theory from the unbounded perspective}\label{s:unbdd}
Let us consider two $\si$-unital $C^*$-algebras $A$ and $B$. In this section we are concerned with the even $KK$-group $KK_0(A,B)$ in the sense of Kasparov, \cite{Kas:OFE}, but we are approaching this $KK$-group by applying the unbounded point of view. The fundamental notion in unbounded $KK$-theory is given in the next definition which is due to Baaj and Julg, \cite{BaJu:TBK}. Before reading this definition, the reader is advised to review the construction of the Lipschitz algebra from Definition \ref{d:lipschitz}. 

\begin{dfn}\label{d:unbkas}
  An \emph{even unbounded Kasparov module} from $A$ to $B$ is given by a triple $( X,\pi,D)$, where
  \begin{enumerate}
  \item $X$ is a countably generated $\zz/2\zz$-graded Hilbert $C^*$-module over $B$;
  \item $\pi : A \to \B L(X)$ is an even $*$-homomorphism ($A$ has the trivial grading);
  \item $D : \T{Dom}(D) \to X$ is an odd selfadjoint and regular unbounded operator acting on $X$.
  \end{enumerate}
  These three ingredients are required to satisfy the following conditions:
  \begin{enumerate}
  \item The \emph{Lipschitz algebra} 
    \[
\T{Lip}_D(A) := \big\{ a \in A \mid \pi(a) \in \T{Lip}_D(X) \big\}
\]
is norm-dense in $A$. 
\item The resolvent of $D$ is \emph{locally compact} in the sense that $\pi(a) \cd (i + D)^{-1}$ is a compact operator on $X$ for all $a \in A$.
  \end{enumerate}
  We say that $(X,\pi,D)$ is \emph{compact}, if the resolvent $(i + D)^{-1}$ is a compact operator on $X$. And, if both $A$ and $\pi : A \to \B L(X)$ are unital, then $(X,\pi,D)$ is said to be \emph{unital}. 

  The odd selfadjoint and regular unbounded operator $D$ is referred to as the \emph{abstract Dirac operator}.
\end{dfn}

We specify that a $\zz/2\zz$-grading of a Hilbert $C^*$-module $X$ is given in terms of a selfadjoint unitary operator which is often denoted by $\ga : X \to X$ (and referred to as the grading operator). Part of the above definition then says that
\[
D \ga = - \ga D  \, \, \T{ and } \, \, \, \pi(a) \ga = \ga \pi(a) \, \, \T{ for all } a \in A .
\]

We continue with a discussion of the Lipschitz algebra $\T{Lip}_D(A)$ introduced in Definition \ref{d:unbkas}. The closed $*$-derivation $d : \T{Lip}_D(X) \to \B L(X)$ introduced after Definition \ref{d:lipschitz} induces a closed $*$-derivation $d : \T{Lip}_D(A) \to \B L(X)$ satisfying that $d(a)$ agrees with the unique bounded adjointable extension of the commutator
\[
[D , \pi(a)] : \T{Dom}(D) \to X .
\]
Remark that $d(a) : X \to X$ anti-commutes with the grading operator $\ga$. We turn $\T{Lip}_D(A)$ into a Banach $*$-algebra where the norm of an element $a \in \T{Lip}_D(A)$ is given by $\| a \|_1 := \| a \| + \| d(a) \|_\infty$.

Let us record the following well-known result and refer to \cite[Proposition 8.12]{Val:IBC} for a detailed argument, see also \cite[Proposition 3.12]{BlCu:DNS}:

\begin{prop}\label{p:local}
The norm-dense $*$-subalgebra $\T{Lip}_D(A) \su A$ satisfies the conditions $(1)$ and $(2)$ stated in the beginning of Section \ref{s:ktheory}. In particular, we get that the inclusion $\io : \T{Lip}_D(A) \to A$ induces an isomorphism of $K$-theory groups. 
\end{prop}

It is sometimes convenient to work with a slightly different notion of an even unbounded Kasparov module defined as follows:

\begin{dfn}\label{d:unbkasII}
  An \emph{even unbounded Kasparov module} over $B$ is given by a pair $( X,D)$, where
\begin{enumerate}
\item $X$ is a countably generated $\zz/2\zz$-graded Hilbert $C^*$-module over $B$;
\item $D : \T{Dom}(D) \to X$ is an odd selfadjoint and regular unbounded operator satisfying that $(i + D)^{-1}$ belongs to $\B K(X)$.
\end{enumerate}
\end{dfn}

Clearly, if $(X,\pi,D)$ is a compact even unbounded Kasparov module from $A$ to $B$, then $(X,D)$ is an even unbounded Kasparov module over $B$. Conversely, if $(X,D)$ is an even unbounded Kasparov module over $B$ and $\pi : A \to \B L(X)$ is an even $*$-homomorphism such that the Lipschitz algebra $\T{Lip}_D(A) \su A$ is norm-dense, then $(X,\pi,D)$ is a compact even unbounded Kasparov module from $A$ to $B$.

In fact, if $(X,D)$ is an even unbounded Kasparov module over $B$ we may consider the Lipschitz algebra $\T{Lip}_D(X)$ from Definition \ref{d:lipschitz} and single out the corresponding even part which we denote by
\begin{equation}\label{eq:evenlip}
\T{Lip}_D^{\T{ev}}(X) := \big\{ T \in \T{Lip}_D(X) \mid T \ga = \ga T \big\} .
\end{equation}
We view $\T{Lip}_D^{\T{ev}}(X)$ as a unital Banach $*$-subalgebra of $\T{Lip}_D(X)$. It is relevant to record that the closure of $\T{Lip}_D^{\T{ev}}(X)$ with respect to the operator norm is a unital $C^*$-subalgebra of $\B L(X)$ and we denote it by $C_D^{\T{ev}}(X)$. Writing $i : C_D^{\T{ev}}(X) \to \B L(X)$ for the inclusion we therefore get an even unital unbounded Kasparov module $(X,i,D)$ from $C_D^{\T{ev}}(X)$ to $B$ and the corresponding Lipschitz algebra agrees with $\T{Lip}_D^{\T{ev}}(X)$.
\medskip

Let us proceed with a brief discussion of the relationship between unbounded Kasparov modules and $KK$-theory in the sense of Kasparov, \cite{Kas:OFE}. For a detailed account of $KK$-theory we refer to the books \cite{Bla:KOA,JeTh:EKT}.

There are various definitions of equivalence relations on the even Kasparov modules from $A$ to $B$. However, in the case where $A$ is separable and $B$ is $\si$-unital they all give rise to the same abelian group, see e.g. \cite[Theorem 2.2.17 and Theorem 5.2.4]{JeTh:EKT}. For definiteness we apply the \emph{operator homotopy} picture of even $KK$-theory where two even Kasparov modules $(X_0,\phi_0,F_0)$ and $(X_1,\phi_1,F_1)$ from $A$ to $B$ are equivalent, if the following holds:
\begin{enumerate}
\item There exist two degenerate even Kasparov modules $(Y_0,\psi_0,G_0)$ and $(Y_1,\psi_1,G_1)$ (again from $A$ to $B$) and an even unitary operator
  \[
U : X_0 \op Y_0 \to X_1 \op Y_1
\]
which intertwines the relevant $*$-homomorphisms, meaning that
\[
U ( \phi_0(a) \op \psi_0(a) ) U^* = \phi_1(a) \op \psi_1(a) \q \T{for all } a \in A ;
\]
\item There exists an operator norm continuous map $H : [0,1] \to \B L(X_0 \op Y_0)$ satisfying that $(X_0 \op Y_0, \phi_0 \op \psi_0,H_t)$ is an even Kasparov module for all $t \in [0,1]$ and that
  \[
H_0 = F_0 \op G_0 \, \, \T{ and } \, \, \, H_1 = U^* (F_1 \op G_1) U .
  \]
\end{enumerate}
Notice that an application of \cite[Proposition 17.4.3]{Bla:KOA} shows that we may work exclusively with even Kasparov modules $(X,\phi,F)$ satisfying that $F = F^*$ and $\| F \|_\infty \leq 1$. Moreover, in the case where $A$ is unital, it suffices to look at unital even Kasparov modules $(X,\phi,F)$ (meaning that $\phi$ is unital).
\medskip

In the special case where $(X,\phi,F)$ is a unital even Kasparov module from $\B C$ to $B$, we are going to suppress $\phi$ from the notation, writing $(X,F)$ instead of $(X,\phi,F)$. A similar convention applies to unital even unbounded Kasparov modules from $\B C$ to $B$.
\medskip
%

For an even unbounded Kasparov module $(X,\pi,D)$ from $A$ to $B$ we obtain from \cite[Proposition 2.2]{BaJu:TBK} that the triple $\big(X,\pi,D(1 + D^2)^{-1/2} \big)$ is an even Kasparov module  from $A$ to $B$, see also \cite[Proposition 17.11.3]{Bla:KOA} for a more detailed proof. Recall in particular from the discussion after Proposition \ref{p:trans} that the \emph{bounded transform} $F_D = D(1 + D^2)^{-1/2} : X \to X$ is indeed a well-defined selfadjoint contraction. It is common to refer to the class $[X,\pi,F_D]$ in $KK_0(A,B)$ as the \emph{Baaj-Julg bounded transform} of $(X,\pi,D)$.

The main result of \cite{BaJu:TBK} is quoted here below (focusing on the case where $A$ and $B$ are both trivially graded):

\begin{thm}\label{t:baajjulg}
  Suppose that $A$ is separable and $B$ is $\si$-unital. Let $(X,\pi,F)$ be an even Kasparov module from $A$ to $B$ with $F = F^*$ and $F^2 = 1$. Then there exists a compact even unbounded Kasparov module of the form $(X,\pi,D)$ such that the identity
$[X,\pi,F_D] = [X,\pi,F]$ is valid in the $KK$-group $KK_0(A,B)$. In particular, it holds that the Baaj-Julg bounded transform is surjective. 
\end{thm}

It is in fact possible to describe the kernel of the Baaj-Julg bounded transform in a sensible way by studying homotopies of unbounded Kasparov modules, see the papers \cite{Kaa:UKK,DuMe:HEU}. These investigations are however less important for our current purpose and we therefore refrain from discussing the details here.

\subsection{The unbounded Kasparov product (easy case)}
Our aim is now to provide a special but very useful version of the unbounded Kasparov product. The unbounded Kasparov product was pioneered by Mesland in \cite{Mes:UCN} and soon afterwards it was further developed in \cite{KaLe:LGR,KaLe:SFU}. Other significant contributions to the general theory can be found in \cite{MeRe:NST,LeMe:SRS}. In the present text we are only dealing with a particularly simple version of the unbounded Kasparov product and it is therefore unnecessary to use the general machinery referred to above. The strategy is however the same in so far that we first find a candidate for the Kasparov product at the unbounded level and then verify that it satisfies the criteria found by Kucerovsky in \cite{Kuc:KUM}. Remark that Kucerovsky's approach to computing Kasparov products relies on the earlier work of Connes and Skandalis, see \cite{CoSk:LIF}. 


\begin{prop}\label{p:unbddprodI}
  Suppose that $(X,D)$ is an even unbounded Kasparov module over $B$ and that $p$ is a projection in $M_n\big( \T{Lip}^{\T{ev}}_D(X) \big)$ for some $n \in \B N$. Then it holds that $\big( p  X^{\op n}, p D^{\op n} p \big)$ is a unital even unbounded Kasparov module from $\B C$ to $B$, where we specify that $p D^{\op n} p$ has domain defined by $\T{Dom}( p D^{\op n} p ) := p \cd \T{Dom}(D)^{\op n}$ and that $p  X^{\op n}$ has grading operator induced by $\ga^{\op n} : X^{\op n} \to X^{\op n}$. 
\end{prop}
\begin{proof}
  We focus on showing that $p D^{\op n} p$ is selfadjoint and regular and has compact resolvent. To this end, it suffices to show that $p D^{\op n} p + (1 - p) D^{\op n} (1 - p) : \T{Dom}(D^{\op n}) \to X^{\op n}$ is selfadjoint and regular and has compact resolvent. However, since $(i + D)^{-1}$ is compact by assumption we get that the selfadjoint and regular unbounded operator $D^{\op n} : \T{Dom}(D)^{\op n} \to X^{\op n}$ has compact resolvent. The same is therefore true for every perturbation of $D^{\op n}$ of the form $D^{\op n} + T$, where $T : X^{\op n} \to X^{\op n}$ is a selfadjoint bounded operator. Indeed, we get from the Kato-Rellich theorem, see \cite[Theorem 4.5]{KaLe:LGR}, that $D^{\op n} + T$ is selfadjoint and regular (on the domain $\T{Dom}(D)^{\op n}$) and compactness of the resolvent then follows from the resolvent identity. The result of the proposition is now obtained by remarking that the difference
  \[
D^{\op n} - \big( p D^{\op n} p + (1 - p) D^{\op n} (1 - p) \big) : \T{Dom}(D^{\op n}) \to X^{\op n}
\]
extends to the selfadjoint bounded operator
\[
(1 - p )d(p) - d(p) (1 - p) : X^{\op n} \to X^{\op n} ,  
\]
where $d : M_n\big( \T{Lip}_D(X) \big) \to \B L(X^{\op n})$ is the closed $*$-derivation introduced around \eqref{eq:derivation}.
\end{proof}

We are now ready to present the promised computation of the Kasparov product by means of unbounded $KK$-theoretic data. Let us emphasize that we are focusing on a very special case of the Kasparov product where one of the Hilbert $C^*$-modules appearing is finitely generated projective.


\begin{prop}\label{p:unbddprodII}
  Suppose that $(X,D)$ is an even unbounded Kasparov module over $B$ and that $p$ is a projection in $M_n\big( \T{Lip}^{\T{ev}}_D(X) \big)$ for some $n \in \B N$. Then it holds that
  \begin{equation}\label{eq:kkprod}
\big[p  \cd \big( C^{\T{ev}}_D(X)\big)^{\op n},0\big] \hot_{C^{\T{ev}}_D(X)} [ X,i,F_D] = \big[ p  X^{\op n}, F_{p D^{\op n} p} \big],
  \end{equation}
  where the notation ``$\hot_{C^{\T{ev}}_D(X)}$'' refers to the Kasparov product between $KK_0\big(\B C, C_D^{\T{ev}}(X) \big)$ and $KK_0\big( C_D^{\T{ev}}(X),B\big)$ with values in $KK_0(\B C,B)$.
\end{prop}  
\begin{proof}
  We apply the criteria found by Kucerovsky in \cite[Theorem 13]{Kuc:KUM}. Notice in this respect that we may view $\big( p  \big( C^{\T{ev}}_D(X)\big)^{\op n}, 0\big)$ as a unital even unbounded Kasparov module from $\B C$ to $C^{\T{ev}}_D(X)$ and the Baaj-Julg bounded transform then agrees with the class  $\big[p  \big(C^{\T{ev}}_D(X)\big)^{\op n},0\big]$ inside $KK_0\big(\B C, C^{\T{ev}}_D(X)\big)$. 

Recall that $i : C^{\T{ev}}_D(X) \to \B L(X)$ denotes the inclusion and identify the interior tensor product $p  \big(C^{\T{ev}}_D(X)\big)^{\op n} \hot_i X$ with the Hilbert $C^*$-module $p  X^{\op n}$. The two last conditions in \cite[Theorem 13]{Kuc:KUM} are clearly satisfied since the selfadjoint and regular unbounded operator appearing in the unbounded Kasparov module $\big( p  \big( C^{\T{ev}}_D(X)\big)^{\op n}, 0\big)$ is equal to zero, see also \cite[Lemma 10]{Kuc:KUM}. We therefore focus on the first condition in \cite[Theorem 13]{Kuc:KUM}. To this end, let $S \in \T{Lip}_D^{\T{ev}}(X)$ and consider a vector of the form $x = p \cd  e_j S$ inside $p \big(C_D^{\T{ev}}(X)\big)^{\op n}$. This vector gives rise to the bounded adjointable operator
\[
T_x : X \to p X^{\op n} \q T_x(\xi) := p \cd e_j S(\xi)
\]
with adjoint given by $T_x^*( \sum_{k = 1}^n \eta_k e_k) = S^*(\eta_j)$. For every vector $\xi \in \T{Dom}(D)$ it then holds that $T_x(\xi)$ belongs to $\T{Dom}(p D^{\op n} p)$ and the relevant ``commutator'' is given by
\[
\begin{split}
p D^{\op n} p \cd T_x (\xi) - T_x \cd D(\xi) & = p D^{\op n} p \cd e_j S(\xi) 
- p \cd e_i S D(\xi) \\
& = p d(p) \cd e_j S(\xi) + p \cd  e_j d(S)(\xi) ,
\end{split}
\]
where the closed $*$-derivation $d : M_n\big( \T{Lip}_D(X) \big) \to \B L(X^{\op n})$ is described near \eqref{eq:derivation}. This shows that the ``commutator''
\[
p D^{\op n} p \cd T_x - T_x \cd D : \T{Dom}(D) \to p X^{\op n}
\]
extends to a bounded adjointable operator from $X$ to $p X^{\op n}$. Notice also that for every vector $\eta \in \T{Dom}(p D^{\op n} p)$ it holds that $T_x^*(\eta)$ belongs to $\T{Dom}(D)$. This shows that the first condition in \cite[Theorem 13]{Kuc:KUM} is satisfied as well and hence that the identity in \eqref{eq:kkprod} is correct. 
\end{proof}





\section{The index isomorphism}\label{s:index}
Let us fix a $\si$-unital $C^*$-algebra $B$ and spend a bit of time describing the isomorphism of abelian groups $K_0(B) \cong KK_0(\B C,B)$, which we refer to as the \emph{index isomorphism}, see \cite[Proposition 17.5.5]{Bla:KOA} and \cite[\S 6]{Kas:OFE}. Our description differs from the treatment given in \cite{Bla:KOA} in so far that we do not rely on the six term exact sequence nor on the triviality of the $K$-theory of $\B L\big( \ell^2(\B N,B) \big)$. The approach presented here is in fact more in line with the original approach in \cite[\S 6]{Kas:OFE} (one may however argue that Kasparov's construction is even more elementary since it does not use excision).

To ease the notation here below, we put
\[
\B K_B := \B K\big( \ell^2(\B N,B) \big) \, \, \T{ and } \, \, \, \B L_B := \B L\big( \ell^2(\B N,B) \big) .
\]
Recall from \cite[Lemma 4]{Kas:HSV} that the $C^*$-algebra of compact operators on $\ell^2(\B N,B)$ can be identified with the (minimal) tensor product of $C^*$-algebras:
\[
\B K_B \cong B \ot \B K\big( \ell^2(\B N)\big) .
\]
Since $K$-theory is stable under tensor products with the compact operators $\B K\big( \ell^2(\B N)\big)$, \cite[Proposition 6.4.1]{RLL:IKC}, we get an isomorphism of abelian groups $K_0(B) \cong K_0( \B K_B )$. Applying excision in $K$-theory, see Theorem \ref{t:excision}, it moreover holds that $K_0(\B L_B,\B K_B) \cong K_0(\B K_B)$ and it therefore suffices to identify the $KK$-group $KK_0(\B C,B)$ with the $K$-group $K_0( \B L_B,\B K_B)$. 

Let $(X,F)$ be a unital even Kasparov module from $\B C$ to $B$ satisfying that $F = F^*$ and $\| F \|_\infty \leq 1$. We are now going to construct the corresponding index class in $K_0(B)$.

Let us denote the grading operator on $X$ by $\ga : X \to X$ and consider the two projections $\ga_+ := (1 + \ga)/2$ and $\ga_- := (1 - \ga)/2$. Define the selfadjoint unitary operator 
\[
U_F := F + \ga \cd \sqrt{1 - F^2} : X \to X 
\]
together with the associated projection
\[
Q_F := U_F \cd \ga_+ \cd U_F : X \to X . 
\]
Remark that modulo the compact operators on $X$ it holds that $U_F$ agrees with the selfadjoint contraction $F$ and that $Q_F$ agrees with the projection $\ga_-$. As a consequence, we get that $(\ga_-,Q_F)$ is a projection in $D\big( \B L(X),\B K(X)\big)$ and we therefore obtain a class
\[
  [ \ga_-,Q_F] - [\ga_-,\ga_-] \in K_0\big( \B L(X), \B K(X) \big) .
  \]
  For later use, we provide a more manageable expression for the projection $Q_F$ which can be verified by a few algebraic manipulations (using that $F$ anti-commutes with $\ga$):
  \begin{equation}\label{eq:expproj}
  Q_F = \ga (1 - F^2) + F \sqrt{1 - F^2} + \ga_- .
  \end{equation}
Since $X$ is countably generated as a Hilbert $C^*$-module over $B$, an application of Kasparov's stabilization theorem allows us to choose an adjointable isometry
$V : X \to \ell^2(\B N,B)$, see \cite[Theorem 2]{Kas:HSV} (or alternatively \cite[Theorem 6.2]{Lan:HCM}). In particular, we obtain the $*$-homomorphism
\[
\T{Ad}(V) : \B L(X) \to \B L_B \q \T{Ad}(V)(T) := V T V^* 
\]
which also satisfies that $\T{Ad}(V)(K) \in \B K_B$ for all $K \in \B K(X)$. Applying functoriality in relative $K$-theory we get the following:

%

\begin{dfn}\label{d:index}
  The \emph{index class} of the unital even Kasparov module $(X,F)$ from $\B C$ to $B$ (with $F = F^*$ and $\| F \|_\infty \leq 1$) is the class in
  $K_0(\B L_B,\B K_B) \cong K_0(B)$ given by the formal difference of projections
  \[
\T{Index}( X,F) := K_0\big( \T{Ad}(V) \big)\big( [\ga_-,Q_F] - [\ga_-,\ga_-] \big) . 
\]
\end{dfn}

We notice in passing that the vanishing of the $K$-group $K_0(\B L_B)$, see \cite[Proposition 12.2.1]{Bla:KOA}, implies that $K_0\big( D(\B L_B,\B K_B)\big) = K_0(\B L_B,\B K_B)$ and hence that the class $[\ga_-,\ga_-] \in K_0\big( D(\B L_B,\B K_B) \big)$ is in fact trivial. This is however less important for our current treatment of the index isomorphism.  
\medskip

Using the operator homotopy picture of $KK$-theory described in Section \ref{s:unbdd}, the proof of the next proposition is rather elementary and has therefore been omitted.

\begin{prop}\label{p:indexiso}
The index class $\T{Index}(X,F) \in K_0(B)$ only depends on the class of the unital even Kasparov module $(X,F)$ in the $KK$-group $KK_0(\B C,B)$. In particular, we obtain a group homomorphism
\[
\T{Index} : KK_0(\B C,B) \to K_0(B) .
\]
\end{prop}

Let now $p = (p_1,p_2)$ be a projection belonging to the matrix algebra $M_n\big( D(\B L_B, \B K_B) \big)$ for some $n \in \B N$. Since the two projections $p_1$ and $p_2$ inside $M_n(\B L_B)$ coincide modulo $M_n(\B K_B)$ we get an even unital Kasparov module from $\B C$ to $B$ of the form
\begin{equation}\label{eq:kaspmod}
\Big( p_2 \ell^2(\B N,B)^{\op n} \op p_1 \ell^2(\B N,B)^{\op n} , \ma{cc}{0 & p_2 p_1 \\ p_1 p_2 & 0} \Big) ,
\end{equation}
where the corresponding grading operator is given by $\ga = 1 \op (-1)$. It is then not difficult to verify that the above construction yields a group homomorphism from $K_0\big( D(\B L_B,\B K_B) \big)$ to $KK_0(\B C,B)$ and upon suppressing the isomorphism between $K_0(B)$ and $K_0(\B L_B,\B K_B) \su K_0\big( D(\B L_B,\B K_B) \big)$ we therefore get a group homomorphism 
\[
\T{Index}^{-1} : K_0( B) \to  KK_0(\B C,B) .
\]


\begin{thm}\label{t:index}
The two group homomorphisms $\T{Index}$ and $\T{Index}^{-1}$ are each others inverse and the abelian groups $KK_0(\B C,B)$ and $K_0(B)$ are thereby isomorphic. 
\end{thm}
\begin{proof}
  First of all, let $(X,F)$ be a unital even Kasparov module from $\B C$ to $B$ with $F = F^*$ and $\| F \|_\infty \leq 1$. Denoting the grading operator on $X$ by $\ga$, let us argue that the class $[X,F]$ in $KK_0(\B C,B)$ agrees with the class coming from the unital even Kasparov module
  \begin{equation}\label{eq:inverseI}
  \Big( Q_F X \op \ga_- X , \ma{cc}{ 0 & Q_F \ga_- \\ \ga_- Q_F & 0} \Big)
  \end{equation}
  This will in turn imply that the composition $\T{Index}^{-1} \ci \T{Index} : KK_0(\B C,B) \to KK_0(\B C,B)$ is equal to the identity. Using the selfadjoint unitary $U_F \in \B L(X)$ and the fact that $Q_F = U_F \ga_+ U_F$ we see that the unital even Kasparov module in \eqref{eq:inverseI} is in fact unitarily equivalent to the unital even Kasparov module
  \[
\Big( \ga_+ X \op \ga_- X, \ma{cc}{ 0 & \ga_+ U_F \ga_- \\ \ga_- U_F \ga_+ & 0 } \Big) .
\]
But it clearly holds that $\ga_+ U_F \ga_- = \ga_+ F \ga_-$ and similarly that $\ga_- U_F \ga_+ = \ga_- F \ga_+$ implying that the unital even Kasparov module in \eqref{eq:inverseI} is in fact unitarily equivalent to $(X,F)$.

Consider now a projection $p = (p_1,p_2)$ in the matrix algebra $M_n\big( D(\B L_B,\B K_B) \big)$ for some $n \in \B N$. Recall that the difference $p_1 - p_2$ belongs to $M_n(\B K_B)$ and define the selfadjoint unitary
\[
U_p := \ma{cc}{ \sqrt{1- p_2 p_1 p_2} & p_2 p_1 \\ p_1 p_2 & - \sqrt{1 - p_1 p_2 p_1}} 
\]
in $M_{2n}(\B L_B)$. We then record that
\[
\begin{split}
& \T{Index}\Big[ p_2 \ell^2(\B N,B)^{\op n} \op p_1 \ell^2(\B N,B)^{\op n}, \ma{cc}{0 & p_2 p_1 \\ p_1 p_2 & 0} \Big] \\
  & \q = \big[ 0 \op p_1, U_p (p_2 \op 0) U_p \big] - [p_1, p_1] .
\end{split}
\]
Thus, in order to see that the composition $\T{Index} \ci \T{Index}^{-1}$ agrees with the identity on $K_0(B) \cong K_0(\B L_B,\B K_B)$ it suffices to show that $[p] = \big[0 \op p_1, U_p (p_2 \op 0) U_p \big]$ inside the $K$-group $K_0\big( D(\B L_B,\B K_B) \big)$. But this follows immediately by noting that the projections $(p_1 \op 0, p_2 \op 0)$ and $\big(0 \op p_1, U_p(p_2 \op 0) U_p \big)$ are Murray-von Neumann equivalent via the partial isometry $\Big( \ma{cc}{0 & 0 \\ p_1 & 0}, U_p (p_2 \op 0) \Big)$ which indeed belongs to $M_{2n}\big( D(\B L_B,\B K_B) \big)$.
\end{proof}

It is relevant to compute the inverse of the index map in the case where the input comes from a projection $p \in M_n(B^\sim)$. To this end, introduce the unital $*$-homomorphism
\begin{equation}\label{eq:scalar}
s : B^\sim \to B^\sim \q s(b,\la) := (0,\la) .
\end{equation}
Next, identify $B$ with the compact operators on $B$ viewed as a Hilbert $C^*$-module over itself. In this fashion, we get a $*$-homomorphism
\begin{equation}\label{eq:compacts}
j : B \to \B L(B)
\end{equation}
and hence, as described in Subsection \ref{ss:standing}, we obtain a unital $*$-homomorphism $j : M_n(B^\sim) \to \B L(B^{\op n})$ for every $n \in \B N$ (upon identifying $M_n\big( \B L(B) \big)$ with $\B L(B^{\op n})$). Viewing $B^{\op n}$ as a closed Hilbert $C^*$-module of $\ell^2(\B N,B)$ (corresponding to the subset $\{1,\ldots,n\} \su \B N$) we also get the $*$-homomorphism $j_n : M_n(B^\sim) \to \B L_B$. It can then be verified that the isomorphism $K_0(B) \cong K_0( \B L_B,\B K_B)$ is induced by the map which sends a projection $p \in M_n(B^\sim)$ to the projection $\big( j_n( s(p)), j_n(p) \big) \in D\big( \B L_B,\B K_B\big)$. Using this description in combination with \eqref{eq:kaspmod} and Theorem \ref{t:index} we obtain the following:
%
%
%
\begin{lemma}\label{l:indexinverse}
  Let $p \in M_n(B^\sim)$ be a projection. We have the identity
  \[
\T{Index}^{-1}\big( [p] - [s(p)] \big) = \Big[ j(p) B^{\op n} \op j( s(p)) B^{\op n} , \ma{cc}{0 & j\big( p s(p)\big) \\ j\big(s(p)p\big) & 0} \Big] .
  \]
  inside the $KK$-group $KK_0(\B C,B)$.
\end{lemma}

\subsection{Naturality}
In this subsection we review the naturality property of the index isomorphism. As far as we can tell, a proof of this property is not contained in Blackadar's book, see \cite[Proposition 17.5.5]{Bla:KOA}, and for the sake of completeness we therefore present some details on it here. 

Consider two $\si$-unital $C^*$-algebras $B$ and $C$ together with a $*$-homomorphism $\si : B \to C$. Using functoriality of $K$-theory and $KK$-theory we get two group homomorphisms
\[
K_0(\si) : K_0(B) \to K_0(C) \, \, \T{ and } \, \, \, \si_* : KK_0(\B C,B) \to KK_0(\B C,C) .
\]
We emphasize that the proof of naturality is less obvious when $\si : B \to C$ fails to be non-degenerate, thus when
\[
\si(B) C := \T{span}_{\B C}\big\{ \si(b) c \mid b \in B \, \, \T{ and } \, \, \, c \in C \big\} 
\]
fails to be norm-dense in $C$. In this case, we have been unable to find a proof which only uses the operator homotopy picture of $KK$-theory (as described in Section \ref{s:unbdd}) and our proof therefore eventually relies on \cite[Theorem 2.2.17]{JeTh:EKT}.

\begin{prop}\label{p:naturality}
  The following diagram is commutative:
  \[
  \begin{CD}
    KK_0(\B C,B) @>{\si_*}>> KK_0(\B C,C) \\
    @V\T{Index}VV @V{\T{Index}}VV \\
K_0(B) @>{K_0(\si)}>> K_0(C)
  \end{CD}
  \]
\end{prop}
\begin{proof}
  Let us view $C$ as a Hilbert $C^*$-module over itself and let $X \su C$ denote the Hilbert $C^*$-submodule obtained as the norm-closure of $\si(B) C$ inside $C$. Suppressing the $*$-homomorphism $j : C \to \B L(C)$, see \eqref{eq:compacts}, we get a $*$-homomorphism $\si : B \to \B L(C)$ which in turn restricts to another $*$-homomorphism $\si_0 : B \to \B L(X)$ (satisfying that $\si_0(b)(x) = \si(b)(x)$ for all $x \in X$). 

  For a projection $p \in M_n(B^\sim)$ we define the $\zz/2\zz$-graded Hilbert $C^*$-module $E_1 := \si(p) C^{\op n} \op \si(s(p))C^{\op n}$ together with the $\zz/2\zz$-graded Hilbert $C^*$-submodule $E_0 := \si_0(p) X^{\op n} \op \si_0(s(p)) X^{\op n} \su E_1$. Using Lemma \ref{l:indexinverse} we get that 
  \begin{equation}\label{eq:naturalI}
    \big( \T{Index}^{-1} K_0(\si) \big)\big( [p] - [s(p)] \big)
  = \Big[ E_1 , \ma{cc}{0 & \si\big(p s(p)\big) \\ \si\big(s(p) p\big) & 0 } \Big]
  \end{equation}
  and, on the other hand, it can be verified that
  \begin{equation}\label{eq:naturalII}
  \big( \si_* \T{Index}^{-1}) \big)\big( [p] - [s(p)] \big)
  = \Big[ E_0 , \ma{cc}{0 & \si\big(p s(p)\big) \\ \si\big(s(p) p\big) & 0} \Big] .
  \end{equation}
  
 Let now $E$ denote the $\zz/2\zz$-graded Hilbert $C^*$-module over $C([0,1],C)$ consisting of all norm-continuous maps $f$ from $[0,1]$ to $E_1$ satisfying that $f(0) \in E_0$ and define the selfadjoint bounded operator $F : E \to E$ by putting
\[
F(f)(t) = \ma{cc}{0 & \si\big(p s(p) \big) \\ \si\big(s(p) p\big) & 0} f(t) \q \T{for all } t \in [0,1] .
\]
This yields a unital even Kasparov module $(E,F)$ from $\B C$ to $C( [0,1], C)$ which in turn constitutes a homotopy between the unital even Kasparov modules
\[
\big( E_0, \ma{cc}{0 & \si\big(p s(p)\big) \\ \si\big(s(p) p\big) & 0} \big) \, \, \T{ and } \, \, \,
\big( E_1, \ma{cc}{0 & \si\big(p s(p)\big) \\ \si\big(s(p) p\big) & 0} \big) .
\]
Appealing to \cite[Theorem 2.2.17]{JeTh:EKT}, this shows that the classes in $KK_0(\B C,C)$ appearing in \eqref{eq:naturalI} and \eqref{eq:naturalII} do in fact coincide. 
\end{proof}

\subsection{The index isomorphism for the complex numbers}
We end this section by revisiting the index isomorphism $KK_0(\B C,\B C) \cong K_0(\B C) \cong \B Z$ in view of the conventions applied in the present text.


\begin{prop}\label{p:indexC}
  Let $\big( G_+ \op G_- ,\ma{cc}{ 0 & F_0^* \\ F_0 & 0} \big)$ be a unital even Kasparov module from $\B C$ to $\B C$ with grading operator $\ga = \ma{cc}{1 & 0 \\ 0 & - 1}$. It holds that $F_0 : G_+ \to G_-$ is a Fredholm operator and upon identifying $K_0(\B C)$ with $\B Z$ we have the identities
  \[
  \T{Index}(F_0) = \T{Dim}_{\B C}\big( \T{Ker}(F_0) \big) - \T{Dim}_{\B C}\big( \T{Ker}(F_0^*) \big)
  = \T{Index}\big(  G_+ \op G_- ,\ma{cc}{ 0 & F_0^* \\ F_0 & 0} \big) .
  \]
\end{prop}
\begin{proof}
  To ease the notation, put $G := G_+ \op G_-$ and $F := \ma{cc}{ 0 & F_0^* \\ F_0 & 0}$. Consider the polar decomposition $F_0 = v|F_0|$ and record that the unital even Kasparov modules $( G, F)$ and $\big( (1 - v^* v) G_+ \op (1 - vv^*) G_-, 0 \big)$ determine the same class in $KK_0(\B C,\B C)$. Hence, with $n := \T{Dim}_{\B C}\big( \T{Ker}(F_0) \big)$ and $m := \T{Dim}_{\B C}\big( \T{Ker}(F_0^*) \big)$ we get from Lemma \ref{l:indexinverse} that
  \[
  \T{Index}( G,F) = \T{Index}( \B C^n \op \B C^m, 0) = [1_n] - [1_m] .
  \]
  This proves the proposition since the isomorphism between $K_0(\B C)$ and $\B Z$ is given in terms of the non-normalized trace $\T{TR}$ (sending the unit $1_k \in M_k(\B C)$ to $k$).
\end{proof}

\section{Spectral localizers}\label{s:spectral}
Throughout this section, we consider an even unbounded Kasparov module $(X,D)$ over a $\si$-unital $C^*$-algebra $B$, see Definition \ref{d:unbkasII}. The grading operator on $X$ is denoted by $\ga : X \to X$. We emphasize that the selfadjoint and regular unbounded operator $D$ is assumed to have compact resolvent. On top of this data, we suppose that $H$ is a selfadjoint invertible element in $\T{Lip}^{\T{ev}}_D(X)$, where we recall the definition of the even part of the Lipschitz algebra $\T{Lip}_D^{\T{ev}}(X) \su \T{Lip}_D(X)$ from \eqref{eq:evenlip} and Definition \ref{d:lipschitz}.  

Our aim is now to provide a Hilbert $C^*$-module version of the even spectral localizer introduced in \cite{LoSc:SLE}. One of the main differences with the Hilbert space setting described in \cite{LoSc:SLE} is that the spectral projections for $D$ are not in general available in the context of Hilbert $C^*$-modules. We therefore need to find a version of the spectral localizer which only uses the continuous functional calculus described in Section \ref{s:func}. The terminology ``spectral localizer'' is however still appropriate since our definition only depends on the part of the spectrum of $D$ which is contained in an interval of the form $[-\rho,\rho]$ for a fixed value of $\rho > 0$. 

Let us present the details.

Since $D$ anti-commutes with $\ga$ we get that the spectrum of $D$ is invariant under change of sign so that
\[
\si(D) = \big\{ - \la \mid \la \in \si(D) \big\} = \si( - D) .
\]
For $f \in C_b\big( \si(D) \big)$ we therefore get a well-defined bounded continuous function $f_{-1} : \si(D) \to \B C$ given by $f_{-1}(x) = f(-x)$. Moreover, the uniqueness property of the continuous functional calculus $\Psi : C_b\big( \si(D)\big) \to \B L(X)$ (see Theorem \ref{t:func}) in combination with Proposition \ref{p:trans} show that
\begin{equation}\label{eq:funcgrad}
\ga f(D) \ga = f(-D) = f_{-1}(D) .
\end{equation}

\begin{dfn}\label{d:fctlocal}
We say that an even bounded continuous $L^1$-function $\phi : \B R \to [0,1]$ is \emph{localizing}, if it satisfies the following three conditions:
\begin{enumerate}
\item The support of $\phi$ is contained in $[-1,1]$ and $\phi(x) = 1$ for all $x \in [-1/2,1/2]$;
\item $\phi(x_0) \geq \phi(x_1)$ whenever $0 \leq x_0 \leq x_1$;
\item $p \cd \widehat{\phi}(p)$ is an $L^1$-function on $\B R$.
\end{enumerate}
\end{dfn}

Remark that it is possible to choose a localizing function $\phi : \B R \to [0,1]$ satisfying that $\| p \cd \widehat{\phi}(p) \|_1 \leq 8 \cd \sqrt{2 \pi}$. A detailed construction of such a localizing function is provided in \cite[Lemma 4]{LoSc:FVC}. 
\medskip

Let us fix a general localizing function $\phi : \B R \to [0,1]$. 

For every strictly positive real number $\rho > 0$ we define the smooth even function $\phi_\rho(x) := \phi(x/\rho)$ and record that $\T{supp}(\phi_\rho) \su [-\rho,\rho]$ whereas $\phi_\rho(x) = 1$ for all $x \in [-\rho/2,\rho/2]$. To ease the notation, we put
\[
\Phi_\rho := \phi_\rho(D) .
\]
Since $\phi_\rho$ is even it follows from \eqref{eq:funcgrad} that $\Phi_\rho$ commutes with the grading operator $\ga$. Moreover, for $\rho \geq \rho' > 0$ we have the inequalities
\[
1 \geq \Phi_\rho \geq \Phi_{\rho'} \geq 0
\]
inside the unital $C^*$-algebra $\B L(X)$. 


\begin{dfn}\label{d:localizer}
   For every $\ka > 0$ and $\rho > 0$, we define the \emph{spectral localizer} as the selfadjoint element 
  \[
  L_{\ka,\rho}(H,D,\phi)
  := \Phi_\rho \ga H \Phi_\rho + \ka \cd \Phi_{2 \rho} D \Phi_{2 \rho}
  - \big(1 - \Phi_{2 \rho}^4 \big)^{1/2} \cd \ga
  \]
  inside the unital $C^*$-algebra $\B L(X)$.
\end{dfn}  

Since the triple $(H,D,\phi)$ is currently fixed we simply write $L_{\ka,\rho}$ instead of $L_{\ka,\rho}(H,D,\phi)$. We immediately record the following result:

\begin{lemma}\label{l:modcomp}
  The spectral localizer $L_{\ka,\rho}$ is equal to $-\ga$ modulo $\B K(X)$. In particular, it holds that $(-\ga, L_{\ka,\rho})$ belongs to $D\big(\B L(X),\B K(X)\big)$.
\end{lemma}
\begin{proof}
Recall that $D$ has compact resolvent. Since the localizing function $\phi : \B R \to [0,1]$ has compact support we get that both $\Phi_\rho$ and $\Phi_{2\rho}$ lie in the closed $*$-ideal $\B K(X) \su \B L(X)$ and hence that $L_{\ka,\rho} + \ga \in \B K(X)$.   
\end{proof}

\begin{lemma}\label{l:specsquare}
  We have the identity
  \[
  \begin{split}
  L_{\ka,\rho}^2 & = 1 - \Phi_{2 \rho}^4 + \ka^2 \cd D^2  \Phi_{2 \rho}^4 + \Phi_\rho^2 \cd H^2 \cd \Phi_\rho^2 \\
  & \q + \ka \cd \Phi_\rho \cd d(H) \ga \cd \Phi_\rho
  + \Phi_\rho \cd \big[ \Phi_\rho H , [\Phi_\rho,H] \big] \cd \Phi_\rho .
  \end{split}
  \]
\end{lemma}
\begin{proof}
  Remark first that $\Phi_\rho \cd \Phi_{2 \rho} = \Phi_\rho = \Phi_{2 \rho} \cd \Phi_\rho$ since $\phi_{2 \rho}(x) = 1$ for all $x$ in the support of $\phi_\rho$. Notice also that both $\Phi_\rho$ and $\Phi_{2\rho}$ commute with all of the other operators involved except for $H$. Using that $H$ commutes with $\ga$ and that $D$ anti-commutes with $\ga$, we now compute that
  \[
  \begin{split}
    L_{\ka,\rho}^2
    & =  \big( \Phi_\rho \ga H \Phi_\rho \big)^2 + \ka^2\cd  D^2 \Phi_{2 \rho}^4 + 1 - \Phi_{2 \rho}^4 \\
  & \q + \ka \cd \big( \Phi_\rho \ga H D \Phi_\rho + \Phi_\rho D \ga H \Phi_\rho \big) 
  - \ka \cd \big(D \ga  + \ga  D \big) \cd \Phi_{2 \rho}^2 \big(1 - \Phi_{2 \rho}^4 \big)^{1/2} \\
  & = \big( \Phi_\rho H \Phi_\rho \big)^2 + \ka^2 \cd D^2 \Phi_{2 \rho}^4 + 1 - \Phi_{2 \rho}^4
  + \ka \cd \Phi_\rho \cd d(H) \ga \cd \Phi_\rho .
  \end{split}
  \]
  The lemma is now proved by noting that
  \[
  \begin{split}
    & H \Phi_\rho^2 H = \Phi_\rho H^2 \Phi_\rho +  [H,\Phi_\rho] \Phi_\rho H + \Phi_\rho H [\Phi_\rho, H] \\
    & \q = \Phi_\rho H^2 \Phi_\rho +  \big[\Phi_\rho H, [\Phi_\rho,H] \big]  . \qedhere
  \end{split}
  \]
\end{proof}

Let us define the constant
\begin{equation}\label{eq:localconst}
c_\phi := \frac{2 \cd \| p \cd \widehat{\phi}(p) \|_1}{\sqrt{2 \pi}} \in (0,\infty)
\end{equation}
which only depends on our fixed choice of localizing function $\phi : \B R \to [0,1]$. For $\ka,\rho \in (0,\infty)$ we then introduce the quantity
\begin{equation}\label{eq:constant}
  C_{\ka,\rho} := \Big(\ka + \frac{c_\phi \cd \| H \|_\infty}{\rho} \Big)\cd \| d(H) \|_\infty \in [0,\infty) .
\end{equation}

We define the \emph{gap} of the selfadjoint invertible bounded operator $H : X \to X$ as the strictly positive number
\[
g := \| H^{-1} \|_\infty^{-1} > 0 .
\]

\begin{lemma}\label{l:square}
  Let $\ka,\rho \in (0,\infty)$. We have the estimate
  \[
  \begin{split}
    L_{\ka,\rho}^2 & \geq 1 - \Phi_{2\rho}^4 + \big( \frac{\ka^2\rho^2}{4} - C_{\ka,\rho} \big) \cd (\Phi_{2\rho}^4 - \Phi_{\rho}^4)
    + ( g^2 - C_{\ka,\rho} ) \cd \Phi_{\rho}^4 
  \end{split}
  \]
\end{lemma}
\begin{proof}
To begin with, record that the bounded operator
  \[
\ka \cd d(H) \ga + \big[ \Phi_\rho H , [\Phi_\rho,H] \big] : X \to X
\]
is selfadjoint and that an application of Corollary \ref{c:scaleest} yields the operator norm estimate
\[
\begin{split}
& \big\| \ka \cd d(H) \ga + \big[ \Phi_\rho H , [\Phi_\rho,H] \big] \big\|_\infty 
\leq \ka \cd \| d(H) \|_\infty + 2  \| H \|_\infty \cd \big\| [\Phi_\rho,H] \big\|_\infty \\
  & \q \leq \ka \cd \| d(H) \|_\infty +  \frac{2  \| H \|_\infty \cd \| p \cd \widehat{\phi}(p) \|_1}{\rho \cd \sqrt{2 \pi}} \cd \| d(H) \|_\infty
  = C_{\ka,\rho} .
\end{split}
\]
These observations imply that
\begin{equation}\label{eq:rest}
\ka \cd \Phi_\rho \cd d(H) \ga \cd \Phi_\rho + \Phi_\rho \cd \big[ \Phi_\rho H , [\Phi_\rho,H] \big] \cd \Phi_\rho
\geq - C_{\ka,\rho} \cd \Phi_\rho^2 .
\end{equation}

It moreover holds that
\begin{equation}\label{eq:extrarest}
\Phi_\rho^2 \cd H^2 \cd \Phi_\rho^2 \geq \| H^{-2} \|^{-1}_\infty \cd \Phi_\rho^4 = g^2 \cd \Phi_\rho^4 .
\end{equation}

Next, notice that the support of the continuous function $\phi_{2 \rho}^4 - \phi_{\rho}^4 : \B R \to [0,1]$ is contained in the union $[-2\rho,-\rho/2] \cup [\rho/2,2\rho]$. Hence, using \eqref{eq:spectrum} and Proposition \ref{p:core} we see that the spectrum of the selfadjoint bounded operator $\big(D^2 - (\rho/2)^2 \big) \cd (\Phi_{2 \rho}^4 - \Phi_{\rho}^4)$ is contained in the closed interval $\big[0, (2\rho)^2 -  (\rho/2)^2\big]$. This entails the inequalities
  \begin{equation}\label{eq:dsquare}
D^2 \cd \Phi_{2 \rho}^4 \geq  D^2  \cd (\Phi_{2 \rho}^4 - \Phi_{\rho}^4) \geq (\rho/2)^2 \cd (\Phi_{2 \rho}^4 - \Phi_{\rho}^4) .
  \end{equation}

  Remark that the selfadjoint bounded operator $\Phi_{2\rho}^4 - \Phi_\rho^2 = \Phi_{2\rho}^2 \cd ( \Phi_{2 \rho}^2 - \Phi_\rho^2)$ is positive. 
  
Combining the inequalities in \eqref{eq:rest}, \eqref{eq:extrarest} and \eqref{eq:dsquare} with the identity from Lemma \ref{l:specsquare} it is now possible to deduce that
 \[
 \begin{split}
   L_{\ka,\rho}^2 & \geq 1 - \Phi_{2\rho}^4 + \ka^2 \cd D^2 \cd \Phi_{2 \rho}^4 + g^2 \cd \Phi_\rho^4 - C_{\ka,\rho} \cd \Phi_\rho^2 \\
   & \geq  1 - \Phi_{2\rho}^4 + \frac{\ka^2 \rho^2}{4} \cd (\Phi_{2 \rho}^4 - \Phi_{\rho}^4) 
  + g^2 \cd \Phi_\rho^4 - C_{\ka,\rho} \cd \Phi_\rho^2 \\
   & \geq 1 - \Phi_{2\rho}^4 + \big( \frac{\ka^2 \rho^2}{4} - C_{\ka,\rho} \big) \cd (\Phi_{2 \rho}^4 - \Phi_{\rho}^4) \\
  & \q + ( g^2 - C_{\ka,\rho} ) \cd \Phi_\rho^4 
   + C_{\ka,\rho} \cd ( \Phi_{2 \rho}^4 - \Phi_\rho^2 ) \\
  & \geq 1 - \Phi_{2\rho}^4 + \big( \frac{\ka^2 \rho^2}{4} - C_{\ka,\rho} \big) \cd (\Phi_{2 \rho}^4 - \Phi_{\rho}^4) 
   + ( g^2 - C_{\ka,\rho} ) \cd \Phi_\rho^4 . \qedhere
  \end{split}
  \]
\end{proof}


\begin{dfn}\label{d:admiss}
  We say that the pair $(\ka,\rho)$ of strictly positive numbers is \emph{admissible} with respect to $(H,D,\phi)$, if it holds that
  \[
C_{\ka,\rho} < \min\{ g^2, \frac{\ka^2 \rho^2}{4} \} .
\]
\end{dfn}

\begin{prop}\label{p:specinv}
  If the pair $(\ka,\rho)$ is admissible, then the spectral localizer $L_{\ka,\rho}$ is invertible. In this case, we get a class $\big[ \big( -\ga, L_{\ka,\rho}(H,D,\phi) \big)\big]$ in the $K$-group $K_0^{\T{inv}}\big(\B L(X),\B K(X)\big)$.
\end{prop}
\begin{proof}
  The constraints on $C_{\ka,\rho}$ in combination with Lemma \ref{l:square} entail that there exists an $\ep \in (0,1]$ such that
  \[
  L_{\ka,\rho}^2 \geq (1 - \Phi_{2\rho}^4) + \ep \cd  \Phi_{2\rho}^4  \geq \ep .
  \]
  This shows that $L_{\ka,\rho}$ is invertible. The last statement of the proposition is now a consequence of Lemma \ref{l:modcomp}. 
\end{proof}


Notice that if $(\ka, \rho)$ is admissible, then it holds that $\ka  \cd \| d(H) \|_\infty < g^2$ and that $(\ka, \rho')$ is admissible for all $\rho' \geq \rho$.


We now discuss how to arrange that $(\ka,\rho)$ becomes admissible. First of all, notice that if $\| d(H) \|_\infty = 0$, then $C_{\ka,\rho} = 0$ and all pairs of strictly positive constants are admissible. The next proposition also takes care of the case where $\| d(H) \|_\infty > 0$.

\begin{prop}\label{p:specinvII}
Let $\ka \in (0,\infty)$ satisfy that $\ka \cd \| d(H) \|_\infty < g^2$.  If $\rho \in (0,\infty)$ satisfies that
  \[
\rho \geq \frac{2g}{\ka} \, \, \mbox{ and } \, \, \, \rho > \frac{c_\phi \cd \| H \|_\infty \cd \|d(H) \|_\infty}{ g^2 - \ka \cd \| d(H) \|_\infty}, 
\]
then the pair $(\ka, \rho)$ is admissible with respect to $(H,D,\phi)$.
\end{prop}
\begin{proof}
  Recall from \eqref{eq:constant} that
\[
C_{\ka,\rho} = \ka \cd \| d(H) \|_\infty + \frac{c_\phi \cd \| H \|_\infty \cd \| d(H) \|_\infty}{\rho}  .
\]
If $\rho > \frac{c_\phi \cd \| H \|_\infty \cd \|d(H) \|_\infty }{g^2 - \ka \cd \| d(H) \|_\infty}$, then we get that
\[
C_{\ka,\rho} < \ka \cd \| d(H) \|_\infty + g^2 - \ka \cd \| d(H) \|_\infty  = g^2 .
\]
Moreover, if $C_{\ka,\rho} < g^2$ and $\rho \geq \frac{2 g}{\ka}$, then we have that
\[
\frac{C_{\ka,\rho}}{\ka^2 \cd \rho^2} \leq \frac{C_{\ka,\rho}}{4 \cd g^2} < \frac{1}{4} .
\]
This shows that the pair $(\ka,\rho)$ is admissible under the constraints stated in the proposition. 
\end{proof}

%

\section{The spectral localizer class and its invariance properties}\label{s:specprop}
Let us fix an even unbounded Kasparov module $(X,D)$ over a $\si$-unital $C^*$-algebra $B$.

Choose a localizing function $\phi : \B R \to [0,1]$ as in Definition \ref{d:fctlocal} and assume that $H \in \T{Lip}_D(X)$ is an even selfadjoint invertible element. According to Proposition \ref{p:specinvII} there exists a pair $(\ka , \rho)$ of strictly positive constants which is admissible with respect to $(H,D,\phi)$ and Proposition \ref{p:specinv} therefore provides us with the corresponding \emph{spectral localizer class}
\[
\big[ \big(-\ga, L_{\ka,\rho}(H,D,\phi)\big) \big] \in K_0^{\T{inv}}\big( \B L(X),\B K(X) \big) .
\]

The aim of this section is to establish a couple of invariance properties for this $K$-theory class. 

\begin{lemma}\label{l:locindepI}
  The spectral localizer class
  \[
\big[ \big(-\ga, L_{\ka,\rho}(H,D,\phi)\big) \big] \in K_0^{\T{inv}}\big( \B L(X),\B K(X) \big) 
\]
is independent of the choice of localizing function $\phi$ and of the choice of admissible pair $(\ka,\rho)$. 
\end{lemma}
\begin{proof}
  Let us start out by showing that $\big[ -\ga,L_{\ka,\rho_0}(H,D,\phi) \big] = \big[ -\ga,L_{\ka,\rho_1}(H,D,\phi) \big]$ whenever $(\ka,\rho_0)$ and $(\ka,\rho_1)$ are admissible with respect to $(H,D,\phi)$. This follows by noting that the pair $\big(\ka, (1 - t)\rho_0 + t \rho_1\big)$ is admissible for all $t \in [0,1]$ and that the map $t \mapsto L_{\ka, (1-t) \rho_0 + t \rho_1}(H,D,\phi)$ defines a homotopy between $\big(-\ga,L_{\ka,\rho_0}(H,D,\phi) \big)$ and $\big(-\ga,L_{\ka,\rho_1}(H,D,\phi) \big)$ in the sense of Definition \ref{d:homotopy}. Indeed, this is a consequence of Theorem \ref{t:func} together with the observation that the assignment $\mu \mapsto \phi_\mu$ yields a continuous map from $(0,\infty)$ to $C_b( \B R)$.

  Consider now two constants $\ka_0, \ka_1 \in (0, \infty)$ with $\ka_0 \leq \ka_1$ and with $\ka_1 \cd \| d(H) \|_\infty < g^2$, where $g = \| H^{-1} \|^{-1}_\infty > 0$ is the gap of $H$. According to Proposition \ref{p:specinvII} we may choose a $\rho \in (0,\infty)$ such that $(\ka,\rho)$ is admissible for all $\ka \in [\ka_0,\ka_1]$. We then get that the map $t \mapsto L_{ (1-t) \ka_0 + t \ka_1,\rho}(H,D,\phi)$ provides a homotopy between $\big(-\ga,L_{\ka_0,\rho}(H,D,\phi) \big)$ and $\big( -\ga,L_{\ka_1,\rho}(H,D,\phi) \big)$. 

  Finally, let $\phi$ and $\psi : \B R \to [0,1]$ be two localizing functions and notice that $(1 - t) \cd \phi + t \cd \psi$ is a localizing function for all $t \in [0,1]$. Consulting \eqref{eq:localconst} we see that $c_{ (1 - t) \cd \phi + t \cd \psi} \leq (1 - t) \cd c_\phi + t \cd c_\psi$. Thus, given $\ka \in (0,\infty)$ with $\ka \cd \| d( H) \|_\infty < g^2$ and using Proposition \ref{p:specinvII}, we may choose a $\rho \in (0,\infty)$ such that $(\ka,\rho)$ is admissible with respect to $\big(H,D, (1 - t) \cd \phi + t \cd \psi \big)$ for all $t \in [0,1]$. In particular, we get the identity $\big[ -\ga,L_{\ka,\rho}(H,D,\phi) \big] = \big[-\ga, L_{\ka,\rho}(H,D,\psi) \big]$ in the $K$-group $K_0^{\T{inv}}\big( \B L(X),\B K(X) \big)$.
\end{proof}

Because of Lemma \ref{l:locindepI} we apply the notation $L(H,D) \in K_0^{\T{inv}}\big(\B L(X), \B K(X) \big)$ for the spectral localizer class. We thus have that
\[
L(H,D) = \big[ \big(-\ga, L_{\ka,\rho}(H,D,\phi)\big) \big]
\]
for every localizing function $\phi : \B R \to [0,1]$ and for every pair $(\ka,\rho)$ which is admissible with respect to $(H,D,\phi)$.

\begin{lemma}\label{l:locindepII}
Let $\al : [0,1] \to  \T{Lip}_D^{\T{ev}}(X)$ be a continuous map with $\al(t)$ selfadjoint and invertible for all $t \in [0,1]$. Then it holds that $L\big( \al(0),D\big)  = L\big( \al(1),D\big)$ as an identity of classes in $K_0^{\T{inv}}\big(\B L(X), \B K(X) \big)$. 
\end{lemma}
\begin{proof}
Put $g_t := \| \al(t)^{-1} \|_\infty^{-1}$ for all $t \in [0,1]$ and let $\phi : \B R \to [0,1]$ be a localizing function. Choose a $\ka \in (0,\infty)$ such that $\ka  \cd \big\| d( \al(t)) \big\|_\infty < g_t^2$ for all $t \in [0,1]$. Applying Proposition \ref{p:specinvII} we may find a $\rho \in (0,\infty)$ such that $(\ka,\rho)$ is admissible with respect to $\big( \al(t),D,\phi\big)$ for all $t \in [0,1]$.  We thus get that $\big(-\ga,L_{\ka,\rho}(\al(0),D,\phi)\big)$ and $\big(-\ga,L_{\ka,\rho}(\al(1),D,\phi) \big) \in W_1\big(\B L(X),\B K(X)\big)$ are homotopic in the sense of Definition \ref{d:homotopy}. This proves the lemma. 
\end{proof}

\begin{prop}\label{p:continuity}
  Let $T : X \to X$ be an odd selfadjoint bounded operator. It holds that $(X,D + T)$ is an even unbounded Kasparov module over $B$ and we have the identity
  \[
L(H, D + T) = L(H,D)
\]
inside the $K$-group $K_0^{\T{inv}}\big( \B L(X), \B K(X) \big)$.
\end{prop}
\begin{proof}
  Remark that since $T$ is odd selfadjoint and bounded we get that $(X,D+ t \cd T)$ is an even unbounded Kasparov module over $B$ for all $t \in [0,1]$. This is a well-known result and a proof can be found in \cite[Definition and Proposition 4.8]{Kaa:UKK}, see also the argument given in the proof of Proposition \ref{p:unbddprodI}. It is moreover clear that $H$ belongs to $\T{Lip}_{D + t T}^{\T{ev}}(X)$ and the corresponding derivative agrees with $d(H) + t \cd [T,H]$ for all $t \in [0,1]$.

  Letting $\phi : \B R \to [0,1]$ be a localizing function we may now use Proposition \ref{p:specinvII} to choose a pair $(\ka,\rho)$ of strictly positive numbers which is admissible for $(H, D + t T, \phi)$ for all $t \in [0,1]$. An application of Proposition \ref{p:perturb} then shows that the map $t \mapsto L_{\ka,\rho}(H, D + t T, \phi)$ is continuous with respect to the operator norm on $\B L(X)$ and we therefore get that $\big(-\ga, L_{\ka,\rho}(H, D,\phi) \big)$ and $\big(-\ga, L_{\ka,\rho}(H, D,\phi) \big)$ are homotopic in the sense of Definition \ref{d:homotopy}.
\end{proof}

\section{Computation of the spectral localizer class}\label{s:compu}
Let us consider an even unbounded Kasparov module $(X,D)$ over a $\si$-unital $C^*$-algebra $B$ as well as a selfadjoint invertible element $H \in \T{Lip}_D^{\T{ev}}(X)$. Associated to $H$ we have the mutually orthogonal projections
\[
H_+ := \frac{1 + H |H|^{-1}}{2} \, \, \T{ and } \, \, \, H_- := \frac{1 - H |H|^{-1}}{2}
\]
which both belong to $\T{Lip}_D^{\T{ev}}(X)$. The present section deals with the key computation of this paper showing that the spectral localizer class $L(H,D)$ constructed in Section \ref{s:specprop} can be described in terms of the index class $\T{Index}(H_+ X, F_{H_+ D H_+})$. Notice in this respect that the class $[ H_+ X, F_{H_+ D H_+} ]$ in $KK_0(\B C,B)$ is obtained by applying the Baaj-Julg bounded transform to the unital even unbounded Kasparov module $(H_+ X, H_+ D H_+)$ constructed in Proposition \ref{p:unbddprodI}. Our strategy is to start out by establishing our key identity in the case where $d(H) = 0$ and $H^2 = 1$. The general case then follows easily by an application of Lemma \ref{l:locindepII} and Proposition \ref{p:continuity}.


Let us fix a localizing function $\phi : \B R \to [0,1]$ in the sense of Definition \ref{d:fctlocal} and define the odd smooth function $g : \B R \to [-1,1]$ by putting $g(x) := (1 - \phi^4(x))^{1/2} x |x|^{-1}$ for all $x \in \B R$. Notice that $g(x) = 0$ for all $x \in [-1/2,1/2]$ and that $g^2 = 1 - \phi^4$. An application of the continuous functional calculus yields the odd selfadjoint bounded operator 
\[
G := g(D) \in \B L(X) .
\]
To ease the notation, we put $\Phi := \phi(D)$ and $\Phi_2 := \phi(D/2)$ and define the even selfadjoint bounded operators
\[
A := \Phi H \Phi \, \, \T{ and } \, \, \, C := (1 - \Phi_2^4 )^{1/2} .
\]
Remark that since $\Phi \cd \Phi_2 = \Phi$ we get have that $A \cd C = C \cd A = 0$.
%

\begin{lemma}\label{l:compI}
  Suppose that $d(H) = 0$ and that $H^2 = 1$. It holds that $\ga \cd (A - C) + G$ and $-\ga + G$ are selfadjoint invertible bounded operators on $X$ which are equal to one another modulo the compact operators on $X$. Moreover, we have the identity
  \[
L(H,D) = \big[ \big(-\ga  + G, \ga \cd (A - C) + G\big)\big]
\]
inside the $K$-group $K_0^{\T{inv}}\big( \B L(X), \B K(X) \big)$.
\end{lemma}
\begin{proof}
  Remark first of all that $G^2 = 1 - \Phi^4$ and that $(A - C)^2 = \Phi^4 + 1 - \Phi_2^4$, where the second equality relies on the identities $d(H) = 0$ and $H^2 = 1$. 

For each $t \in [0,1]$, define the odd selfadjoint bounded operator
\[
K_t := (1 - t)^{1/2} \Phi_2 D \Phi_2 + t^{1/2} G .
\]
Remark that $\ga \cd (A -C) + K_1 = \ga \cd (A-C) + G$ and that
\[
\ga \cd (A - C) + K_0 = \ga \cd \big( \Phi H \Phi - (1 - \Phi_2^4)^{1/2} \big) + \Phi_2 D \Phi_2 = L_{1,1}(H,D,\phi),
\]
see Definition \ref{d:localizer}. In particular, we already know from Proposition \ref{p:specinv} that $\ga \cd (A - C) + K_0$ is invertible. We now show that $\ga \cd (A-C) + K_t$ is invertible for $t \in (0,1]$. Indeed, using that $d(H) = 0$ and that all of the bounded operators $1 - \Phi_2^4$, $\Phi_2 D G \Phi_2$ and $D^2 \Phi_2^4$ are positive, we get that
\[
\begin{split}
  & \big(\ga \cd (A - C) + K_t\big)^2 = (A - C)^2 + K_t^2 \\
  & \q = \Phi^4 + 1 - \Phi_2^4 + (1 - t)  D^2 \Phi_2^4 + t G^2 +
2 \cd t^{1/2} (1 - t)^{1/2} \Phi_2 D G \Phi_2  \\
& \q \geq \Phi^4 + t (1 - \Phi^4)
\geq t .
\end{split}
\]
We also remark that $\Phi$ and $\Phi_2$ are both compact, entailing that $\ga \cd ( A - C) + K_t$ agrees with $-\ga + t^{1/2} G$ modulo the compact operators for all $t \in [0,1]$. Since it clearly holds that $-\ga + t^{1/2} G$ is invertible for all $t \in [0,1]$ we obtain a continuous map 
\[
[0,1] \to W_1\big(\B L(X), \B K(X)\big) \q t \mapsto \big( -\ga + t^{1/2} G, \ga \cd (A-C) + K_t \big) 
\]
and hence that $\big(-\ga, L_{1,1}(H,D,\phi)\big)$ and $\big(-\ga + G, \ga \cd (A-C) + G\big)$ are homotopic.
\end{proof}

\begin{lemma}\label{l:compII}
  Suppose that $d(H) = 0$ and that $H^2 = 1$. Let $\mu : [-1,1] \to [-1,1]$ be a continuous function with $\mu(0) \leq 0$ and with $\mu(-1) = -1$ and $\mu(1) = 1$. Suppose moreover that there exists a $\de \in (0,1)$ such that $\mu$ is continuously differentiable on $[-1,-\de]$ and on $[\de,1]$. Let $\nu : [-1,1] \to [0,\infty)$ be the unique continuous function such that $\nu(x) \cd (1 - x^2)^{1/2} = ( 1 - \mu^2(x))^{1/2}$. Then it holds that the pair
\[
\big( -\ga + \nu(0)G,\ga \cd \big( \mu(A) - (1 + \mu(0)) C \big) + \nu( A) G \big)
\]
belongs to $W_1\big( \B L(X),\B K(X)\big)$ and the corresponding class in $K_0^{\T{inv}}\big(\B L(X), \B K(X)\big)$ agrees with the spectral localizer class $L(H,D)$.
\end{lemma}
\begin{proof}
To ease the notation, we put
\[
B_\mu := \mu(A) - (1 + \mu(0)) \cd C 
\]
Since $\mu(0) \cd (1 + \mu(0)) \leq 0$ and $A \cd C = C \cd A = 0$ we get that
\[
B_\mu^2 = \mu^2(A) + (1 + \mu(0))^2 C^2 - 2 \mu(0) (1 + \mu(0)) C  \geq \mu^2(A) .
\]
Since $B_\mu$ is even, $\nu(A) G$ is odd, $d(H) = 0$ and $H^2 = 1$ we then obtain that
\[
\begin{split}
  \big(\ga \cd B_\mu  + \nu(A) G \big)^2 & = B_\mu^2 + \nu^2(A) \cd G^2 = B_\mu^2 + \nu^2(A) \cd (1 - \Phi^4) \\
  & = B_\mu^2 + \nu^2(A) \cd (1 - A^2) \geq 1 .
  \end{split}
\]
This shows that $\ga \cd B_\mu + \nu(A) G$ is a selfadjoint invertible element in $\B L(X)$. It moreover clearly holds that $-\ga + \nu(0) G$ is selfadjoint and invertible. Since $B_\mu$ agrees with $-1$ modulo compacts and $\nu(A) G$ agrees with $\nu(0) G$ modulo compacts we thus get that
\[
\big( -\ga + \nu(0)G, \ga \cd B_\mu + \nu(A) G \big)
\]
belongs to $W_1\big(\B L(X),\B K(X)\big)$.

Our aim is now to show that the class $\big[ \big( -\ga + \nu(0)G, \ga \cd B_\mu + \nu(A) G \big) \big]$ in $K_0^{\T{inv}}\big( \B L(X), \B K(X) \big)$ agrees with the spectral localizer class $L(H,D)$. To this end, we first remark that $\mu(x) = x$ and $\nu(x) = 1$ satisfy the conditions of the present lemma and in this case we get from Lemma \ref{l:compI} that
\[
\big[ ( -\ga + \nu(0)G, \ga \cd B_\mu + \nu(A) G ) \big] = \big[ \big( -\ga + G, \ga \cd (A-C) + G \big) \big] = L(H,D) .
\]

Hence, to establish the result of the lemma it suffices to show that the class $\big[ ( -\ga + \nu(0)G, \ga \cd B_\mu + \nu(A) G ) \big]$ is in fact independent of the choice of continuous function $\mu : [-1,1] \to [-1,1]$. Let thus $\mu_0 : [-1,1] \to [-1,1]$ be an alternative continuous function which satisfies the conditions of the lemma. Choose a $\de \in (0,1)$ such that both $\mu$ and $\mu_0$ are continously differentiable on $[-1,1] \sem (-\de,\de)$. For each $t \in [0,1]$, define the continuous function
    \[
\mu_t := (1 - t) \mu_0 + t \mu
\]
and record that $\mu_t$ also satisfies the conditions in the statement of the lemma (and our chosen $\de \in (0,1)$ of course applies here as well). Let $\nu_t : [-1,1] \to [0,\infty)$ denote the unique continuous function satisfying that $\nu_t(x) (1 - x^2)^{1/2} = ( 1 - \mu_t^2(x) )^{1/2}$. It holds that the map from $[0,1] \to C( [-1,1])$ which sends $t$ to $\mu_t$ is continuous once $C([-1,1])$ is equipped with the supremum norm. The same statement does however not hold in general for the map which sends $t$ to $\nu_t$ and we therefore need to be a bit careful at this point.

    To finish the proof of the lemma it suffices to show that the map $t \mapsto \nu_t(A) G$ is continuous with respect to the operator norm on $\B L(X)$. Notice in this respect that the assignments $t \mapsto B_{\mu_t}$ and $t \mapsto \nu_t(0) G  = ( 1 - \mu_t(0) )^{1/2} G$ clearly depend continuously on $t$. Since $\mu_t$ is continuously differentiable on $[-1,1] \sem (-\de,\de)$ and satisfies that $\mu_t(-1) = -1$ and $\mu_t(1) = 1$ we may choose a constant $C \geq 0$ such that
    \[
    1 - \mu_t^2(x) \leq C \cd (1 - x^2) 
    \]
    for all $x \in [-1,1] \sem (-\de,\de)$ and all $t \in [0,1]$. We thus get that
    \begin{equation}\label{eq:estimatenu}
\nu_t(x) \leq \sqrt{C} \q \T{for all } x \in [-1,1] \sem (-\de,\de) \T{ and } t \in [0,1].
\end{equation}
Notice next that $G$ can be factorized as $(1 - \Phi^4)^{1/4} \cd G_0$ where $G_0$ is another odd selfadjoint and bounded operator (coming from the function $x \mapsto (1 - \phi^4(x))^{1/4} x |x|^{-1}$). Moreover, it holds that $(1 - \Phi^4)^{1/4} = (1 - A^2)^{1/4}$. So we just need to show that $\nu_t(A) (1 - A^2)^{1/4}$ depends continuously on $t$. But this now follows from the estimate in \eqref{eq:estimatenu} and the construction of $\nu_t$. 
\end{proof}

Recall the definition of the isomorphism $\varphi$ from Theorem \ref{t:inviso} and the unbounded $KK$-theoretic constructions from Section \ref{s:unbdd}. 


\begin{prop}\label{p:zerocommu}
Suppose that $d(H) = 0$ and that $H^2 = 1$. For every adjointable isometry $V : X \to \ell^2(\B N,B)$, it holds that
  \[
K_0^{\T{inv}}\big( \T{Ad}(V) \big)\big( L(H,D) \big) = \varphi^{-1}\big(  \T{Index}( H_+ X, F_{H_+ D H_+} ) \big)
\]
inside the $K$-group $K_0^{\T{inv}}( \B L_B, \B K_B )$.
\end{prop}
\begin{proof}
  We start out by applying Lemma \ref{l:compII} in the situation where $\mu : [-1,1] \to [-1,1]$ and $\nu : [-1,1] \to [0,\infty)$ are given by
    \[
    \mu(x) = \fork{ccc}{2x^2 - 1 & \T{for} & x \in [0,1] \\ -1 & \T{for} & x \in [-1,0]} \, \, \T{ and } \, \, \,
    \nu(x) = \fork{ccc}{ 2x & \T{for} & x \in [0,1] \\ 0 & \T{for} & x \in [-1,0] } .
\]
We thus get that
\begin{equation}\label{eq:zeroI}
L(H,D) = \big[ \big(-\ga, \ga \cd \mu( \Phi H \Phi) + \nu( \Phi H \Phi) G \big) \big] .
\end{equation}
Next, since $H^2 = 1$ and $d(H) = 0$, we may identify $X$ with the direct sum $H_+ X \op H_- X$ and under this identification $H$ and $D$ agree with $1 \op (-1)$ and $H_+ D H_+ \op H_- D H_-$, respectively. In particular, we obtain that
\[
\Phi H \Phi = \phi^2(D) H = \phi^2(H_+ D H_+) \op ( - \phi^2(H_- D H_-) ) 
\]
and hence that 
\begin{equation}\label{eq:zeroII}
\begin{split}
  \mu(\Phi H \Phi)  & = ( 2 \phi^4(H_+ D H_+) - 1 ) \op (-1) \q \T{and} \\
  \nu( \Phi H \Phi) G  & = 2 \phi^2( H_+ D H_+) g( H_+ D H_+) \op 0 .
\end{split}
\end{equation}
Let $i_{H_+} : H_+ X \to X$ denote the inclusion (which is indeed an adjointable isometry). By a slight abuse of notation we also write $\ga$ for the grading operator on $H_+ X$ which is induced by the grading operator on $X$. It then follows from \eqref{eq:zeroI} and \eqref{eq:zeroII} that
\[
\begin{split}
L(H,D) = K_0^{\T{inv}}\big( \T{Ad}(i_{H_+} ) \big)
\big[ \big( -\ga, & \ga \cd ( 2 \phi^4(H_+ D H_+) - 1 ) \\ & \q + 2 \phi^2( H_+ D H_+) g( H_+ D H_+) \big) \big] .
\end{split}
\]

Applying the naturality properties from Proposition \ref{p:functoriality} and consulting Definition \ref{d:index} it now suffices to establish the identity
\[
\begin{split}
& \big[ \big( -\ga, \ga \cd ( 2 \phi^4(H_+ D H_+) - 1 ) + 2 \phi^2( H_+ D H_+) g( H_+ D H_+) \big) \big] \\
& \q = \varphi^{-1}\big( [ \ga_-, Q_{F_{H_+ D H_+}}] - [\ga_-,\ga_-] \big)
\end{split}
\]
inside the $K$-group $K_0^{\T{inv}}\big( \B L( H_+ X), \B K(H_+ X) \big)$. To this end, remark that the difference $x(1 + x^2)^{-1/2} - g(x)$ yields an element in $C_0(\B R)$ and this entails that $F_{H_+ D H_+} - g(H_+ D H_+)$ is a compact operator on $H_+ X$. We thus obtain that
\[
[ \ga_-, Q_{F_{H_+ D H_+}}] - [\ga_-, \ga_-] = [ \ga_-, Q_{g(H_+ D H_+)}] - [\ga_-,\ga_-]
\]
inside the $K$-group $K_0\big( \B L(H_+ X), \B K(H_+ X) \big)$. An application of the identity in \eqref{eq:expproj} shows that
\[
Q_{g(H_+ D H_+)} = \ga \phi^4(H_+ D H_+) + g( H_+ D H_+) \phi^2(H_+ D H_+) + \ga_- 
\]
and it therefore holds that
\[
\begin{split}
& \varphi^{-1}\big( [ \ga_-, Q_{g(H_+ D H_+)}] - [\ga_-,\ga_-] \big)
= \big[ \big( 2 \ga_- - 1, 2 Q_{g(H_+ D H_+)} - 1 \big) \big] \\
& \q = \big[ \big( -\ga, \ga \cd ( 2 \phi^4(H_+ D H_+) - 1 ) + 2 \phi^2( H_+ D H_+) g( H_+ D H_+)  \big) \big] .
\end{split}
\]
This proves the present proposition.
\end{proof}

\begin{thm}\label{t:specindex}
For every adjointable isometry $V : X \to \ell^2(\B N,B)$, it holds that
  \[
K_0^{\T{inv}}\big( \T{Ad}(V) \big)\big( L(H,D) \big) = \varphi^{-1}\big( \T{Index}( H_+ X, F_{H_+ D H_+} ) \big)
\]
inside the $K$-group $K_0^{\T{inv}}( \B L_B, \B K_B )$.
\end{thm}
\begin{proof}
  Because of Lemma \ref{l:phase}, Proposition \ref{p:local} and Lemma \ref{l:locindepII} we may assume without loss of generality that $H^2 = 1$. Record then that $D - H_+ d(H_-) - H_- d(H_+) = H_+ D H_+ + H_- D H_-$ and that $T = - H_+ d(H_-) - H_- d(H_+)$ is an odd selfadjoint bounded operator. It therefore follows from Proposition \ref{p:continuity} that
\begin{equation}\label{eq:indexI}
L(H,D) = L\big(H, H_+ D H_+ + H_- D H_-\big) .
\end{equation}
Moreover, since $H$ commutes with $H_+ D H_+ + H_- D H_-$ we get from Proposition \ref{p:zerocommu} that
\begin{equation}\label{eq:indexII}
K_0^{\T{inv}}\big( \T{Ad}(V) \big)\big( L(H,H_+ D H_+ + H_- D H_-) \big) = \varphi^{-1}\big(\T{Index}(H_+ X, F_{H_+ D H_+}) \big) .
\end{equation}
A combination of the two identities in \eqref{eq:indexI} and \eqref{eq:indexII} proves the present theorem.
\end{proof}

\section{Spectral localizers and the index homomorphism}\label{s:specindex}
Let us fix a separable $C^*$-algebra $A$ and a $\si$-unital $C^*$-algebra $B$. It is the aim of this section to establish the main theorem of this paper which amounts to a computation of the Kasparov product
\begin{equation}\label{eq:kasprod}
\hot_A : KK_0(\B C,A) \ti KK_0(A,B) \to KK_0(\B C,B)
\end{equation}
in terms of spectral localizers.

An application of Theorem \ref{t:baajjulg} shows that we may, without loss of generality, fix an even compact unbounded Kasparov module $(X,\pi,D)$ from $A$ to $B$ and consider the associated Baaj-Julg bounded transform $[X,\pi,F_D]$ in $KK_0(A,B)$. Let us denote the corresponding \emph{index homomorphism} between the even $K$-groups by
\[
\binn{ \bullet, [X,\pi,F_D] } : K_0(A) \to K_0(B) .
\]
This index homomorphism is given in terms of the Kasparov product in \eqref{eq:kasprod} and the index isomorphism from Theorem \ref{t:index}. Moreover, because of Proposition \ref{p:local} it suffices to consider an input for the index homomorphism given as a formal difference of the form $[p] - [s(p)]$ where $p$ is a projection belonging to $M_n\big( \T{Lip}_D(A)^\sim \big)$ (see \eqref{eq:scalar} for the definition of $s(p)$). To avoid pathologies we assume that $X \neq \{0\}$.


Recall the definition of the group isomorphism $\varphi : K_0^{\T{inv}}\big( \B L_B, \B K_B) \to K_0(\B L_B,\B K_B)$ from Theorem \ref{t:inviso} and let us suppress the isomorphism $K_0(\B L_B,\B K_B) \cong K_0(B)$, see Theorem \ref{t:excision} and the discussion in the beginning of Section \ref{s:index}.

We are now ready to state and prove the main theorem of this paper.

\begin{thm}\label{t:main}
  Let $(X,\pi,D)$ be an even compact unbounded Kasparov module from $A$ to $B$ and let $p \in M_n\big( \T{Lip}_D(A)^\sim\big)$ be a projection for some $n \in \B N$. For every adjointable isometry $V : X^{\op n} \to \ell^2(\B N,B)$ we have the identity 
  \begin{equation}\label{eq:mainide}
    \begin{split}
    & \varphi K_0^{\T{inv}}\big(\T{Ad}(V) \big)\big( L( 2 \pi(p) - 1, D^{\op n}) - L( 2 \pi(s(p)) - 1, D^{\op n} ) \big) \\
      & \q = \binn{ [p] - [s(p)], [X,\pi,F_D]  }
      \end{split}
  \end{equation}
  inside the $K$-group $K_0(B)$.
\end{thm}
\begin{proof}
  Since $(X,\pi,D)$ is compact we obtain a unital even unbounded Kasparov module $(X,i,D)$ from $C_D^{\T{ev}}(X)$ to $B$. Using the functoriality properties of the Kasparov product, see \cite[Proposition 18.7.1]{Bla:KOA}, and the fact that $\pi^*[X,i,F_D] = [X,\pi,F_D]$ we get the identity
  \begin{equation}\label{eq:func}
    \begin{split}
& \T{Index}^{-1}\big( [p] - [s(p)]\big) \hot_A [X,\pi,F_D] \\
      & \q = \pi_*\big( \T{Index}^{-1}\big( [p] - [s(p)]\big) \big) \hot_{C_D^{\T{ev}}(X)} [X,i,F_D] .
    \end{split}
\end{equation}

By naturality of the index isomorphism, see Proposition \ref{p:naturality}, and the concrete expression for $\T{Index}^{-1}$ from Lemma \ref{l:indexinverse} we then obtain that
\begin{equation}\label{eq:natural}
  \begin{split}
    & \pi_*\big( \T{Index}^{-1}\big( [p] - [s(p)]\big) 
    = \T{Index}^{-1}\big( \big[\pi(p)\big] - \big[\pi(s(p))\big]\big) \big) \\
& \q = \big[ \pi(p) \cd \big(C_D^{\T{ev}}(X)\big)^{\op n}, 0 \big] - \big[ \pi(s(p)) \cd \big(C_D^{\T{ev}}(X)\big)^{\op n}, 0 \big] .
\end{split}
\end{equation}

Remark next that both $\pi(p)$ and $\pi(s(p))$ belong to $M_n\big( \T{Lip}_D(X)\big)$. Hence, combining Proposition \ref{p:unbddprodII} with \eqref{eq:func} and \eqref{eq:natural}, we see that the right hand side of \eqref{eq:mainide} can be rewritten as
\[
\begin{split}
\binn{ [p] - [s(p)], [X,\pi,F_D]  } & =
\T{Index}(\pi(p) X^{\op n}, F_{ \pi(p) D^{\op n} \pi(p)}) \\
& \q - \T{Index}\big( \pi(s(p)) X^{\op n}, F_{ \pi(s(p)) D^{\op n} \pi(s(p))} \big) .
\end{split}
\]

Using that $(X^{\op n},D^{\op n})$ can be viewed as an even unbounded Kasparov module over $B$ and that both $\pi(p)$ and $\pi(s(p))$ can be identified with elements in $\T{Lip}_{D^{\op n}}^{\T{ev}}(X^{\op n})$ we now obtain the result of the present theorem as a consequence of Theorem \ref{t:specindex}. 
\end{proof}

\section{Spectral projections and relationship with earlier results}\label{s:specproj}
Throughout this section we still work with a fixed even unbounded Kasparov module $(X,D)$ over a $\si$-unital $C^*$-algebra $B$. We also assume that $H : X \to X$ is an even selfadjoint invertible operator which belongs to the Lipschitz algebra $\T{Lip}_D(X) \su \B L(X)$. On top of this data, we consider an even projection $P : X \to X$ which is supposed to commute with $D$. In particular, it holds that $P \in \T{Lip}_D(X)$ and that $d(P) = 0$.

Let $\phi : \B R \to [0,1]$ be a localizing function in the sense of Definition \ref{d:fctlocal} and let $\ka$ and $\rho$ be strictly positive real numbers. Recall from Proposition \ref{p:specinv} that if the pair $(\ka,\rho)$ is admissible with respect to $(H,D,\phi)$, then we have the spectral localizer class
\[
L(H,D) = \big[ (- \ga, L_{\ka,\rho}(H,D,\phi) ) \big] \in K_0^{\T{inv}}\big( \B L(X), \B K(X) \big) .
\]

\begin{prop}\label{p:specproj}
  Let $(\ka, \rho)$ be admissible with respect to $(H,D,\phi)$. Suppose that $P : X \to X$ is an even projection which commutes with $D$ and satisfies that $P \Phi_\rho = \Phi_\rho$ and $P \Phi_{2 \rho} = P$. Then the selfadjoint bounded operator
  \[
L_\ka(H,D,P) := \ga \cd \big( P H P - (1 - P) \big) + \ka \cd P D P : X \to X
\]
is invertible and determines a class $\big[ (-\ga, L_\ka(H,D,P) ) \big]$ in $K_0^{\T{inv}}\big( \B L(X), \B K(X) \big)$. Moreover, we have the identity
\[
\big[ (-\ga, L_\ka(H,D,P) ) \big] = \big[ (-\ga, L_{\ka,\rho}(H,D,\phi) ) \big] .
\]
\end{prop}
\begin{proof}
  Remark that $P$ is a compact operator on $X$ with image contained in $\T{Dom}(D)$. This follows since $\Phi_{2\rho} P = P$ and since $D$ has compact resolvent. In particular, we get that $PDP$ is a compact operator which is defined on all of $X$.

  For each $t \in [0,1]$ and $\mu \in \{\rho,2\rho\}$, define the even compact operator
  \[
\Phi_\mu(t) := (1 - t) P + t \Phi_\mu
\]
and notice the following string of inequalities
\begin{equation}\label{eq:string}
\Phi_\rho \leq \Phi_\rho(t) \leq P\leq \Phi_{2\rho}(t) \leq \Phi_{2\rho} \leq 1 .
\end{equation}
Introduce the selfadjoint bounded operator
\[
L_{\ka,\rho}(t) := \Phi_\rho(t) \ga H \Phi_\rho(t) + \ka \cd \Phi_{2\rho}(t) D \Phi_{2\rho}(t) - \ga \cd \big( 1 - ( \Phi_{2\rho}(t) )^4 \big)^{1/2} 
\]
which agrees with $-\ga$ modulo the compact operators on $X$. Notice that the corresponding map from $[0,1]$ to $\B L(X)$ is continuous with respect to the operator norm on $\B L(X)$ and that the endpoints are given by
\[
L_\ka(H,D,P) = L_{\ka,\rho}(0) \, \, \T{ and } \, \, \, L_{\ka,\rho}(H,D,\phi) = L_{\ka,\rho}(1) . 
\]
To obtain the result of the proposition it therefore suffices to show that $L_{\ka,\rho}(t)$ is invertible for all $t \in [0,1]$.

Let thus $t \in [0,1]$ be fixed. Remark first of all that
$P \Phi_{2 \rho}(t) = P$ and $\Phi_\rho \Phi_{2 \rho}(t) = \Phi_\rho$ and hence that $\Phi_\rho(t) \Phi_{2 \rho}(t) = \Phi_\rho(t)$. Following the computation given in the proof of Lemma \ref{l:specsquare} we obtain that
\begin{equation}\label{eq:squaret}
\begin{split}
\big( L_{\ka,\rho}(t) \big)^2
& = \big( \Phi_\rho(t) H \Phi_\rho(t) \big)^2  + \ka^2 \cd  D^2 \big( \Phi_{2\rho}(t) \big)^4 + 1 - \big( \Phi_{2\rho}(t) \big)^4 \\
& \q + \ka \cd \Phi_\rho(t) \cd d(H) \ga \cd \Phi_\rho(t) .
\end{split}
\end{equation}
Referring one more time to the proof of Lemma \ref{l:specsquare} as well as \eqref{eq:string} we also get that
\[
H \big( \Phi_\rho(t) \big)^2 H \geq H \Phi_\rho^2 H = \Phi_\rho H^2 \Phi_\rho + \big[ \Phi_\rho H, [\Phi_\rho,H] \big] .
\]
Putting $g := \| H^{-1} \|_\infty^{-1}$, the proof of Lemma \ref{l:square} and \eqref{eq:string} now entail that
\begin{equation}\label{eq:estiA}
\begin{split}
& \big( \Phi_\rho(t) H \Phi_\rho(t) \big)^2 + \ka \cd \Phi_\rho(t) \cd d(H) \ga \cd \Phi_\rho(t) \\
& \q \geq \Phi_\rho(t) \cd \Phi_\rho H^2 \Phi_\rho \cd \Phi_\rho(t) \\
& \qq + \Phi_\rho(t) \cd \big( \big[ \Phi_\rho H, [\Phi_\rho,H] \big] + \ka \cd d(H) \ga \big)  \cd \Phi_\rho(t) \\
& \q \geq g^2 \cd \Phi_\rho^4 - C_{\ka,\rho} \cd \big( \Phi_\rho(t) \big)^2 
\geq g^2 \cd \Phi_\rho^4 - C_{\ka,\rho} \cd P .
\end{split}
\end{equation}

In order to estimate the term $\ka^2 \cd  D^2 \big( \Phi_{2\rho}(t) \big)^4$ from below, remark that
\[
\Phi_{2\rho}(t) - \Phi_\rho = t (\Phi_{2 \rho} - \Phi_\rho) +  (1 - t) ( P - \Phi_\rho)
= (\Phi_{2 \rho} - \Phi_\rho)\big( t + (1 - t) P \big)
\]
and hence that $D^2( \Phi_{2\rho}(t) - \Phi_\rho) \geq \frac{\rho^2}{4} ( \Phi_{2\rho}(t) - \Phi_\rho)$. We thus get that
\begin{equation}\label{eq:estiB}
\ka^2 \cd  D^2 \big( \Phi_{2\rho}(t) \big)^4 \geq \ka^2 \cd D^2\big( \big( \Phi_{2 \rho}(t) \big)^4 - \Phi_\rho^4 \big) 
\geq \frac{\ka^2 \rho^2}{4} \cd \big( \big( \Phi_{2 \rho}(t) \big)^4 - \Phi_\rho^4 \big) .
\end{equation}
A combination of the computations and estimates in \eqref{eq:squaret}, \eqref{eq:estiA} and \eqref{eq:estiB}, now yields that
\begin{equation}\label{eq:estiC}
\big( L_{\ka,\rho}(t) \big)^2 
\geq 1 - \big( \Phi_{2\rho}(t) \big)^4 + g^2 \cd \Phi_\rho^4 - C_{\ka,\rho} \cd P + \frac{\ka^2 \rho^2}{4} \cd \big( \big( \Phi_{2 \rho}(t) \big)^4 - \Phi_\rho^4 \big) .
\end{equation}
Finally, using that the pair $(\ka,\rho)$ is admissible we may choose an $\ep \in (0,1]$ such that $g^2 \geq \ep + C_{\ka,\rho}$ and $\frac{\ka^2 \rho^2}{4} \geq \ep + C_{\ka,\rho}$. The estimate in \eqref{eq:estiC} therefore entails that
\[
\big( L_{\ka,\rho}(t) \big)^2 \geq 1 - \big( \Phi_{2\rho}(t) \big)^4 + \ep \cd \big( \Phi_{2 \rho}(t) \big)^4 + C_{\ka,\rho} \cd \big( \big(\Phi_{2\rho}(t) \big)^4 - P \big)
\geq \ep > 0 
\]
and we conclude that $L_{\ka,\rho}(t)$ is invertible. 
\end{proof}

\subsection{Relationship with earlier results}
In this last subsection we explain why our results recover the results obtained by Loring and Schulz-Baldes in \cite{LoSc:SLE} in the special case where the $\si$-unital $C^*$-algebra $B$ is equal to $\B C$. We also refer the reader to the recent book \cite{DSW:SF} which, among many other things, contains a detailed argument for the main theorem in \cite{LoSc:SLE} based on spectral flow computations. This approach is in line with the papers \cite{LSS:CHS,ScSt:SLS}. 
%
%


Consider a non-trivial separable $\zz/2\zz$-graded Hilbert space $G$ and let $D : \T{Dom}(D) \to G$ be an odd selfadjoint unbounded operator with compact resolvent. For $\rho > 0$, let $P_\rho := 1_{(-\rho,\rho)}(D)$ denote the spectral projection associated to the open interval $(-\rho,\rho)$ and the abstract Dirac operator $D$. Remark that $P_\rho$ has finite dimensional image. Let $\ga : G \to G$ denote the grading operator and decompose $G$ as $G_+ \op G_-$ where $G_+$ and $G_-$ are the images of the projections $\ga_+$ and $\ga_-$, respectively. Since $D$ is odd and selfadjoint we may identify $D$ with an off-diagonal matrix
\[
D = \ma{cc}{0 & D_0^* \\ D_0 & 0} : \T{Dom}(D_0) \op \T{Dom}( D_0^*) \to G_+ \op G_- .
\]

%

\begin{thm}
  Let $H : G \to G$ be a selfadjoint invertible operator inside $\T{Lip}_D^{\T{ev}}(G)$ and put
\[
Q := 1_{[0,\infty)}(H) = \frac{1 + H |H|^{-1}}{2} .
\]
Let $\phi : \B R \to [0,1]$ be a localizing function and let $(\ka,\rho)$ be an admissible pair with respect to $(H,D, \phi )$. For
\[
P_\rho = 1_{(-\rho,\rho)}(D) \, \, \mbox{ and } \, \, \, Q D Q = \ma{cc}{0 &  (QDQ)_0^* \\ (QDQ)_0 & 0} : Q \T{Dom}(D) \to Q G
\]
it holds that $P_\rho (\ka D + \ga H ) P_\rho$ is a selfadjoint invertible operator on $P_\rho G$ and that $(Q D Q)_0$ is an unbounded Fredholm operator. Moreover, we have the identity
  \[
    \T{Index}\big( (Q D Q )_0 \big) = \frac{1}{2}\T{sign}\big(  P_\rho (\ka D + \ga H ) P_\rho\big) + \frac{1}{2}\T{sign}( \ga P_\rho ) .
  \]
\end{thm}
\begin{proof}
  Remark first of all that $(G,D)$ is an even unbounded Kasparov module over $\B C$ and that $Q$ belongs to $\T{Lip}_D^{\T{ev}}(G)$. It therefore follows from Proposition \ref{p:unbddprodI} that $(Q G, Q D Q)$ is a unital even unbounded Kasparov module from $\B C$ to $\B C$. In particular, we get that $(Q D Q)_0$ is an unbounded Fredholm operator with index equal to the index of the bounded Fredholm operator $(Q D Q)_0\big( 1 + (Q D Q)_0^* (Q D Q)_0\big)^{-1/2}$. We therefore obtain from Proposition \ref{p:indexC} that
  \begin{equation}\label{eq:indfinalI}
    \begin{split}
    \T{Index}\big( (Q D Q )_0 \big) & = \T{Index}\big( Q G, Q D Q \big(1 + (Q D Q)^2\big)^{-1/2} \big) \\
    & = \T{Index}( Q G, F_{Q D Q} ) .
    \end{split}
  \end{equation}

Next, notice that $P_\rho \Phi_\rho = \Phi_\rho$ and that $P_\rho \Phi_{2 \rho} = P_\rho$. This follows since the support of $\phi_\rho$ is contained in $[-\rho,\rho]$ and since $\phi_{2\rho}(x) = 1$ for all $x \in [-\rho,\rho]$. The conditions in Proposition \ref{p:specproj} are therefore satisfied and we get that $P_\rho (\ka D + \ga H ) P_\rho$ is invertible on $P_\rho G$ and that 
  \begin{equation}\label{eq:sigprojI}
  \big[ (-\ga,L_\ka(H,D, P_\rho) )\big] = L(H,D)
  \end{equation}
  inside $K_0^{\T{inv}}\big( \B L(G), \B K(G) \big)$. Applying the notation $i_\rho : P_\rho G \to G$ for the inclusion we moreover have that
  \begin{equation}\label{eq:sigprojII}
  K_0^{\T{inv}}\big( \T{Ad}(i_\rho) \big)\big( \big[ (-\ga P_\rho, P_\rho (\ka D + \ga H) P_\rho) \big] \big)
  = \big[ (-\ga,L_\ka(H,D, P_\rho) )\big] .
  \end{equation}
  Hence, suppressing the isomorphism $K_0^{\T{inv}}\big( \B L(G),\B K(G) \big) \cong \B Z$ we see from \eqref{eq:sigprojI} and \eqref{eq:sigprojII} in combination with Proposition \ref{p:signature} that the spectral localizer class $L(H,D)$ identifies with the integer
  \[
\frac{1}{2}\T{sign}\big(  P_\rho (\ka D + \ga H ) P_\rho\big) + \frac{1}{2}\T{sign}( \ga P_\rho ) .
  \]
  Identifying $K_0(\B C)$ with $\B Z$ we thus obtain from Theorem \ref{t:specindex} that
  \begin{equation}\label{eq:indfinalII}
\T{Index}\big( Q G, F_{Q D Q} \big) = \frac{1}{2}\T{sign}\big(  P_\rho (\ka D + \ga H ) P_\rho\big) + \frac{1}{2}\T{sign}( \ga P_\rho ) .
\end{equation}
The result of the present theorem now follows by combining \eqref{eq:indfinalI} and \eqref{eq:indfinalII}. 
\end{proof}

\bibliographystyle{amsalpha-lmp}

\providecommand{\bysame}{\leavevmode\hbox to3em{\hrulefill}\thinspace}
\providecommand{\MR}{\relax\ifhmode\unskip\space\fi MR }
\providecommand{\MRhref}[2]{%
  \href{http://www.ams.org/mathscinet-getitem?mr=#1}{#2}
}
\providecommand{\href}[2]{#2}

\end{document}